\documentclass[11pt]{article} 
\pdfoutput=1
\usepackage[utf8]{inputenc}
\usepackage[T3,T1]{fontenc}
\usepackage[english]{babel}
\makeatletter\newcommand{\note}[1]{{\textcolor{red}{[#1]}}\@latex@warning{Note: #1}}\makeatother
\usepackage{amsmath,amsfonts,amssymb,amsthm}
\usepackage{enumitem}
\usepackage{mathtools}
\usepackage{booktabs}

\usepackage{multirow}

\DeclareSymbolFont{tipa}{T3}{cmr}{m}{n}
\DeclareMathAccent{\invbreve}{\mathalpha}{tipa}{16}

\usepackage[mono=false]{libertine}
\usepackage[cmintegrals,libertine]{newtxmath}
\usepackage[cal=euler, scr=boondoxo]{mathalfa}

\useosf
\usepackage[a4paper,vmargin={3.5cm,3.5cm},hmargin={2.5cm,2.5cm}]{geometry}
\usepackage[font={small,sf,it}, labelfont={sf,bf}, margin=1cm]{caption}
\usepackage[colorlinks=true]{hyperref}
\usepackage{color,graphicx}
\usepackage{placeins}

\usepackage{stackrel}
\usepackage{soul} 

\usepackage{subcaption}

\theoremstyle{plain}
\newtheorem{theorem}{Theorem}
\newtheorem{corollary}[theorem]{Corollary}
\newtheorem{proposition}[theorem]{Proposition}
\newtheorem{lemma}[theorem]{Lemma}
\theoremstyle{definition}

\newcommand{\ind}{\mathbf{1}}
\newcommand{\R}{\mathbb{R}}

\newcommand{\Z}{\mathbb{Z}}

\newcommand{\rmd}{\mathrm{d}}

\newcommand{\tree}{\mathfrak{t}}
\newcommand{\polytope}{\mathcal{A}_\tree}
\newcommand{\dd}{\mathop{}\!\mathrm{d}}
\newcommand{\vertex}{\mathsf{v}}
\newcommand{\edge}{\mathsf{e}}
\newcommand{\treeset}{\mathfrak{T}}
\newcommand{\vsetinner}{\mathsf{V}}
\newcommand{\vsetboundary}{\mathsf{B}}
\newcommand{\edgeset}{\mathsf{E}}
\newcommand{\cornerset}{\mathsf{C}}
\newcommand{\corner}{\mathsf{c}}
\newcommand{\spinebijection}{\mathsf{Spine}}
\newcommand{\bvertex}{\mathsf{b}}
\newcommand{\dubtree}{\mathsf{t}}

\newcommand{\pdv}[1]{\frac{\partial}{\partial #1}}

\usepackage{tikz}

\newcommand{\hypwedge}{%
  \mathrel{\vcenter{\hbox{%
    \begin{tikzpicture}[scale=0.25, line join=round, line cap=round]
      \coordinate (A) at (0,0.9);
      \coordinate (B) at (1,0.9);
      \coordinate (C) at (0.5,0.);
      \draw[thick]
        (A) .. controls (0.3,0.5) and (0.35,0.6) .. (C)
        (C) .. controls (0.65,0.6) and (0.7,0.5) .. (B)
        (B) .. controls (0.5,1.1) and (0.5,1.1) .. (A);
    \end{tikzpicture}%
  }}}%
}
\newcommand{\hypdiamond}{%
  \mathrel{\vcenter{\hbox{%
    \begin{tikzpicture}[scale=0.29, line join=round, line cap=round]
      \coordinate (A) at (0.1,0.5);
      \coordinate (B) at (0.9,0.5);
      \coordinate (C) at (0.5,0);
	  \coordinate (D) at (0.5,1);
      \draw[thick]
        (A) .. controls (0.25,0.55) and (0.45,0.75) .. (D)
		(A) .. controls (0.25,0.45) and (0.45,0.25) .. (C)
        (B) .. controls (0.75,0.55) and (0.55,0.75) .. (D)
		(B) .. controls (0.75,0.45) and (0.55,0.25) .. (C)
		(A) -- (B);
    \end{tikzpicture}%
  }}}%
}

\makeatletter
\DeclareRobustCommand{\cev}[1]{%
  \mathpalette\do@cev{#1}%
}
\newcommand{\do@cev}[2]{%
  \fix@cev{#1}{+}%
  \reflectbox{$\m@th#1\vec{\reflectbox{$\fix@cev{#1}{-}\m@th#1#2\fix@cev{#1}{+}$}}$}%
  \fix@cev{#1}{-}%
}
\newcommand{\fix@cev}[2]{%
  \ifx#1\displaystyle
    \mkern#23mu
  \else
    \ifx#1\textstyle
      \mkern#23mu
    \else
      \ifx#1\scriptstyle
        \mkern#22mu
      \else
        \mkern#22mu
      \fi
    \fi
  \fi
}
  
\makeatother

\begin{document}
\title{A tree bijection for cusp-less planar hyperbolic surfaces}
\author{Bart Zonneveld}
\maketitle
\begin{abstract}
	Recently, a tree bijection has been found for planar hyperbolic surfaces, which allows for an easy computation of the Weil--Petersson volumes, and opens the path to get distance statistic on random hyperbolic surfaces and to find scaling limits when the number of boundaries becomes large. 
    Crucially, this tree bijection requires the hyperbolic surface to have at least one cusp as origin, from which point distances are measured. 

    In this paper we will extend this tree bijection, such that having a cusp is no longer required. 
    We will first extend the bijection to half-tight cylinders.
    Since general planar hyperbolic surfaces can be naturally decomposed in two half-tight cylinders, this general case is also covered.

    In the half-tight cylinder the distances to the origin are replaced by the so-called Busemann function. 
    This Busemann function is not well-defined on the surface, but it is on the cylinder cover. 
\end{abstract}

\section{Introduction}

Since the ground-breaking work \cite{Mirzakhani2006_SimpleGeodesics_TopRec} by Mirzakhani, there has been a great interest in Weil--Petersson hyperbolic surfaces as a model for random surfaces. 
In this work she finds a recursion relation for the Weil--Petersson volumes. 
It was soon discovered that this recursion is an example of topological recursion \cite{eynard2016_countingSurfaces}, and has interesting applications in physics, in particular for Jackiw--Teitelboim quantum gravity \cite{Mertens2023_Solvable_JTReview}. 

The work also opened up other research directions, focussing on the high genus limit \cite{Mirzakhani2010_Growth_High_genus}, or on finding various statistics of lengths of geodesics on (random) hyperbolic surfaces.

Recently, there are works that re-derive some of Mirzakhani's results using a different approach, such that it can benefit from the existing methods of peeling \cite{Budd2026_Peeling} or such that a bijective correspondence can be found, which can help to find distance statistics. 
A great example of the benefit of a bijective method is \cite{budd2025_randompuncturedhyperbolicsurfaces}, where the authors use a tree bijection to find local and scaling limits of planar hyperbolic surfaces with high number of cusps.

In \cite{Budd2025_Tree_bij} the tree bijection of \cite{budd2025_randompuncturedhyperbolicsurfaces} is further refined, allowing for some geodesic boundaries for the planar hyperbolic surfaces, which gives an interpretation of the `string equation', which is well-known in physics literature (e.g. \cite{DOUGLAS1990_Strings}).
Crucially, the tree bijection in \cite{Budd2025_Tree_bij} depends on the presence of at least one cusp. 
In this work, we will get rid of this restriction and derive a (double) tree bijection for planar hyperbolic surfaces where all boundaries are geodesics of positive length.

It is well known that Mirzakhani's recursion for the Weil--Petersson volumes extends to the case where we have cone-points instead of geodesic boundaries, although one needs to be careful when the cone angles become large, in which case the Weil--Petersson volumes have some non-smooth `wall-crossings' \cite{anagnostou2023_weilpetersson_WallCross}. 
In upcoming work \cite{Budd2026_Cones}, the adaptation of the tree bijection from \cite{Budd2025_Tree_bij} to the case where the boundaries are cones and at least one cusp, will be discussed.
The case where boundaries are cones is still open.  
\\\\
The motivation to search for an extension of the tree bijection from \cite{Budd2025_Tree_bij} in the case where all boundaries are geodesics mainly comes from the connection with maps. 
Here \cite{Budd2025_Tree_bij} should be seen as the equivalent of the Bouttier--Di~Francesco--Guitter bijection \cite{Bouttier2004_Planar_BDFG}, which is a bijection for pointed maps.
In a series of papers \cite{Bouttier2014_OnIrreducible_BGM_A,Bouttier2022_Bijective_BGM_B,Bouttier2024_OnQuasi_BGM_C} Bouttier, Guitter and Miermont, introduces slices to bijectively enumerate tight planar (bipartite) maps.
These papers are the main motivation for this work. 

More specifically, in this paper Theorem~\ref{thm:HTC_bij} is heavily inspired by \cite[Section~7]{Bouttier2014_OnIrreducible_BGM_A} (see also Chapter~2 of \cite{Bouttier2019_Planar_Hab_thesis}, specifically Theorem~2.1), while Theorem~\ref{thm:Full_bij} follows from the same decomposition into half-tight cylinders as proven for maps in \cite[Section~4.4]{Bouttier2024_OnQuasi_BGM_C}. 
\\\\
It is important to note that Theorem~\ref{thm:recursion} of this paper, has been proven in \cite{Budd2024_Top_rec} to hold for any genus $g\geq0$ by solving algebraic equations of generating functions. 
This suggests that the tree bijection from this paper might be extended to general genus surfaces.
Although there is a natural generalization of the tree bijection to $g>0$, which is a bijection to a unicellular map, this fails to be useful in the computation of Weil--Petersson volumes $V_{g,n}(\mathbf{L})$ for $g>0$.
Whether there is a different variant of this tree bijection that can explain the  Weil--Petersson volumes for higher genera is an open question.

\subsection{Half-tight cylinders}
As said, we will extend the tree bijection from \cite{Budd2025_Tree_bij}, which requires at least one cusp (say $L_1=0$) to work. 
Before we can tackle the general case, the case is considered where $L_1>0$, but in some sense still small.

\begin{figure}
    \centering
    \begin{subfigure}[t]{.45\textwidth}
        \centering
        \includegraphics[scale=3]{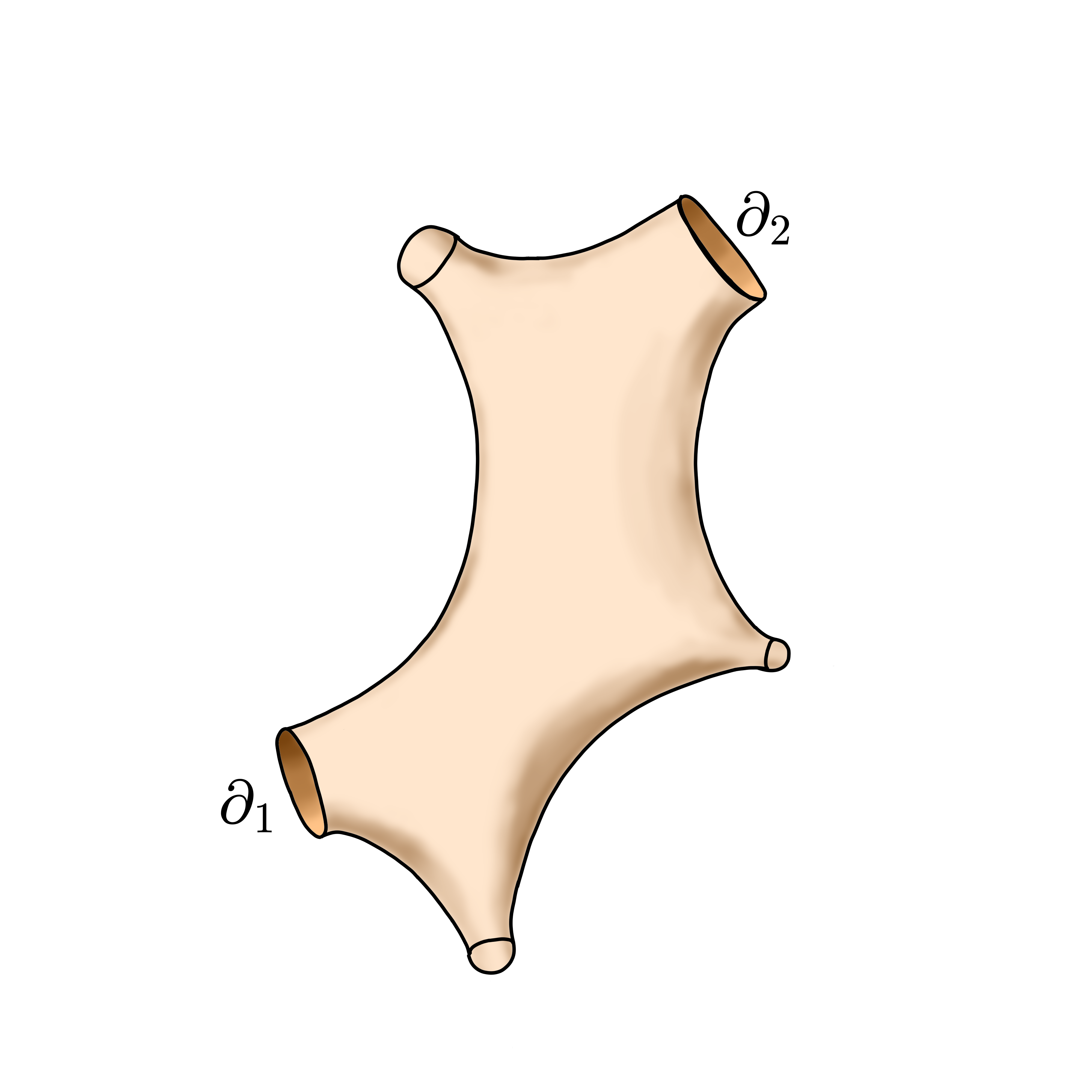}
        \caption{To see whether a curve is separating, we cap off all boundaries except the first two. This gives a cylinder.\label{fig:capped}}
    \end{subfigure}
    \hspace{.05\textwidth}
    \begin{subfigure}[t]{.45\textwidth}
        \centering
        \includegraphics[scale=3]{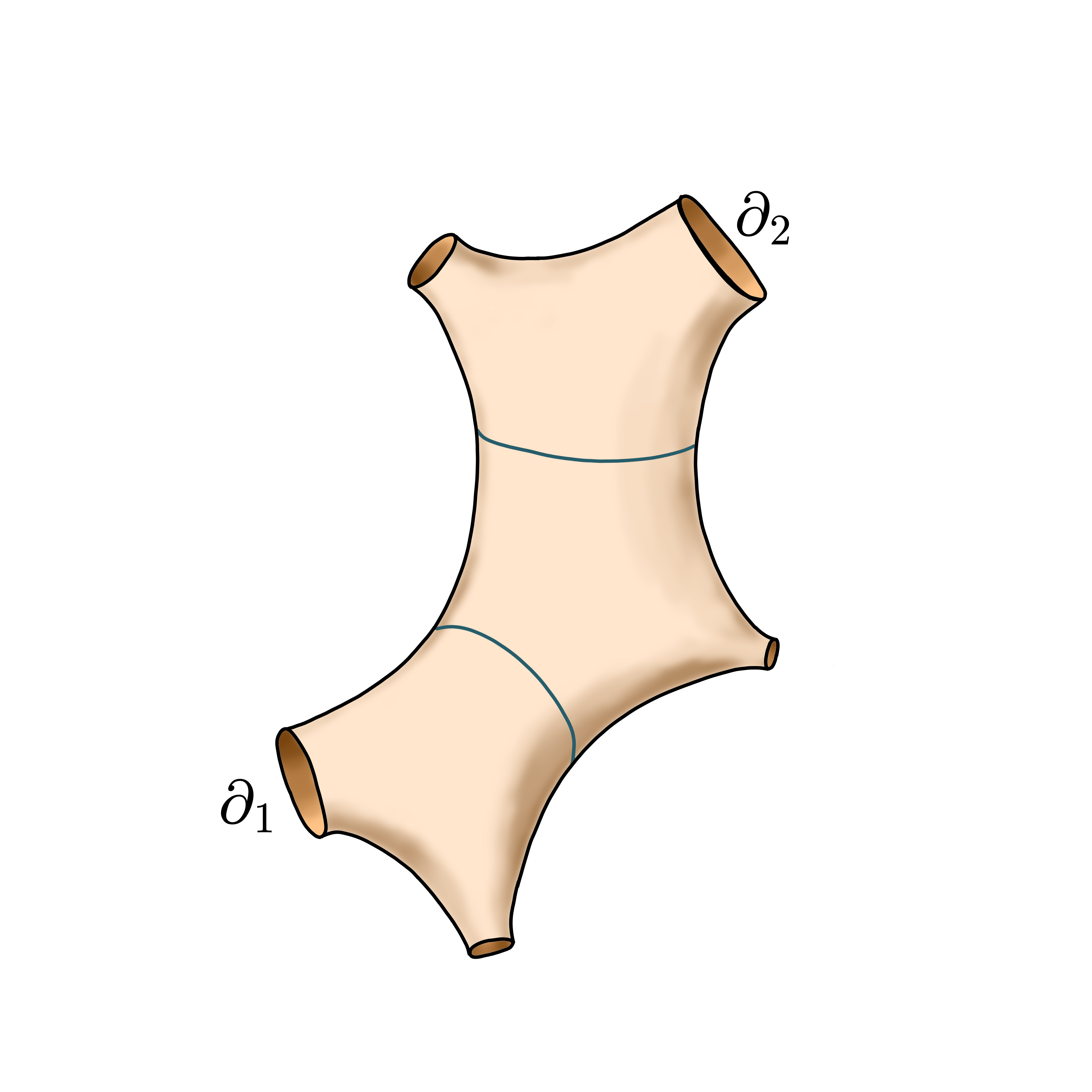}
        \caption{Two separating curves. This surface is a half-tight cylinder if the length of the separating curves (including $\partial_2$) is strictly less than $L_1$.}
    \end{subfigure}
    \caption{}
\end{figure}

To make this statement precise, we interpret a surface $X\in\mathcal{M}_{0,n}(\mathbf{L})$ as a topological cylinder by only considering $\partial_1$ and $\partial_2$ as real boundaries, whereas the other boundaries $\partial_3,\cdots,\partial_n$ should topologically be considered as being `capped off', see Figure~\ref{fig:capped}.%
\footnote{This is the analogous to the distinction between boundaries and inner faces in the setting of maps.}
In this interpretation, the surface is topologically a cylinder, and each closed curve is either null-homotopic or homotopic to (an integer  multiple of) $\partial_1$.
In the latter case the curve is called \emph{separating}.

A surface $X\in\mathcal{M}_{0,n}(\mathbf{L})$ is called a half-tight cylinder when $\partial_1$ is strictly tight, which means that it is the unique shortest separating curve.\footnote{Note that in \cite{Budd2024_Top_rec} it is $\partial_2$ that is tight.}
Note that this implies $L_1<L_2$.

We denote the moduli space of half-tight cylinders by $\mathcal{H}_n(\mathbf{L})\subseteq \mathcal{M}_{0,n}(\mathbf{L})$ and the corresponding Weil--Peterson volumes as $H_n(\mathbf{L})$.
The generating function of these volumes is given by
\begin{align}
 H(L_1,L_2;\mu] &= \sum_{p=1}^\infty \frac{1}{p!} \int \dd{\mu(L_3)}\cdots \int \dd{\mu(L_{2+p})} H_p(\mathbf{L}),  
\end{align}
where $\mu$ is a real linear function on the ring of even, real polynomials, and we use the notation
\begin{align}
    \mu(f) = \int \rmd \mu(K) f(K).
\end{align}
See \cite{Budd2024_Top_rec} for more details.

These half-tight cylinders will be the important building block in this paper. 
We will both show that the tree bijection from \cite{Budd2025_Tree_bij} can be easily extended to half-tight cylinders and that a general planar hyperbolic surface can be decomposed into half-tight cylinders.

\subsubsection*{Trees and polytopes of the half tight cylinder}
We will adapt the notation from \cite{Budd2025_Tree_bij}.
A major difference in notation will be that in \cite{Budd2025_Tree_bij}, surfaces with $n+1$ boundaries are considered, with the first one being the origin and the others labeled $\{1,\ldots,n\}$, while in this paper, we will consider surfaces with $n$ boundaries, where the first boundary gets label $1$ and the remaining boundaries get labels $\{2,\ldots,n\}$.
Also note that we assume all $L_\bvertex>0$.%
\footnote{One could include $L_\bvertex=0$ for some $\bvertex\neq\bvertex_1,\bvertex_2$, using similar techniques as in \cite{Budd2025_Tree_bij}, but then one can also just relabel to get $L_1=0$ and use the bijection in \cite{Budd2025_Tree_bij} entirely.}

That being said, we introduce the set of bicolored (red and white) plane trees $\tree\in\treeset^{\mathrm{all,HTC}}_n$ with $n-1$ white \emph{boundary} vertices $\bvertex\in\vsetboundary(\tree)$ labeled $2,\ldots, n$, and a finite number of red \emph{inner} vertices $\vertex\in\vsetinner$ of degree at least $3$.
Furthermore, each corner of a boundary vertex is labeled either \emph{ideal} or \emph{non-ideal}, the latter being denoted by $\cornerset(\tree)$.
In contrast to \cite{Budd2025_Tree_bij}, all boundary vertices have at least one non-ideal corner.%
\footnote{This is actually in agreement with $\treeset^{\mathrm{all}}_n(\mathbf{L})$ in \cite{Budd2025_Tree_bij}, since we will assume $L_\bvertex>0$.}

The edges in the tree are denoted by $\edge\in\edgeset(\tree)$, or $\vec{\edge},\cev{\edge}\in\vec{\edgeset}(\tree)$ when we use oriented edges.
These oriented edges are also denoted by $\vec{\edge}_{\vertex,i}$, which is the $i$th edge starting in $\vertex$ in clockwise order, where the first edge is chosen in an arbitrary deterministic way.
We have similar notation $\corner_{\bvertex,j}$ for non-ideal corners.
Finally, we denote $\bvertex_i\in\vsetboundary(\tree)$ the boundary vertex with label $i$.

\begin{figure}
    \centering
    \begin{subfigure}[t]{.45\textwidth}
        \centering
        \includegraphics[scale=1]{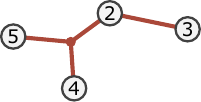}
        \caption{An example tree $\tree\in\treeset^{\mathrm{all,HTC}}_5$. One of the corners of $\bvertex_2$ can be marked as ideal. 
        Note that this tree has a top-dimensional polytope when all corners are marked as non-ideal.}
    \end{subfigure}
    \hspace{.05\textwidth}
    \begin{subfigure}[t]{.45\textwidth}
        \centering
        \includegraphics[scale=1]{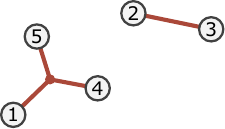}
        \caption{An example double tree $\dubtree\in\treeset^{\mathrm{all,double}}_5$.
        All corners need to be marked as non-ideal. 
        This double tree as a top-dimensional polytope.}
    \end{subfigure}
    \caption{}
\end{figure}

To a $\tree\in\treeset^{\mathrm{all,HTC}}_n$ we associate the \emph{half-tight polytope} $\mathcal{A}_{\tree}(\mathbf{L})$ as follows:
\begin{align}
	\mathcal{A}_{\tree}(\mathbf{L}) \coloneqq &\left\{ \varphi : \vec{\edgeset}(\tree) \to \R_{\geq 0} \middle|\begin{array}{l}
	\varphi(\vec{\edge}_{\bvertex,i}) = 0\text{ for }\bvertex\in\vsetboundary(\tree), 1\leq i\leq\deg(\bvertex),\\
	\varphi(\vec{\edge}_{\vertex,i}) > 0\text{ for }\vertex\in \vsetinner(\tree),  1\leq i\leq\deg(\vertex),\\
	\sum_{i=1}^{\deg(\vertex)}\varphi(\vec{\edge}_{\vertex,i}) = \pi\text{ for }\vertex\in \vsetinner(\tree), \nonumber\\ 
	\varphi(\vec{\edge}) + \varphi(\cev{\edge}) < \pi\text{ for }\edge\in \edgeset(\tree)\end{array}\right\} \\
    &\times \left[\Delta_{\deg(\bvertex_2)}^{((L_2-L_1)/2)} \times \Delta_{\operatorname{nonid}(\bvertex_2)}^{((L_2+L_1)/2)}\right]
    \times \prod_{\bvertex\in \vsetboundary(\tree)\setminus\{\bvertex_2\}} \left[\Delta_{\deg(\bvertex)}^{(L_{\bvertex}/2)} \times \Delta_{\operatorname{nonid}(\bvertex)}^{(L_{\bvertex}/2)}\right].\label{eq:polytopedef}
\end{align}
where $\Delta_k^{(y)} = \{ \mathbf{x} \in (0,\infty)^{k} : x_1 + \cdots + x_{k} = y \}$ is the open $(k-1)$-dimensional simplex of size $y$.

The half-tight polytopes are of dimension 
\begin{align}
	\dim\mathcal{A}_\tree(\mathbf{L}) 
	&= 2n-6 + \sum_{\vertex\in\vsetinner(\tree)} (3-\deg(\vertex)) + \sum_{\bvertex\in\vsetboundary(\tree)} (\operatorname{nonid}(\bvertex)-\deg(\bvertex)), \label{eq:dimpolytope}
\end{align}
from which follows that the top-dimensional polytopes correspond to trees $\tree\in\treeset^{\mathrm{all,HTC}}_n$, where all inner vertices are of degree $3$ and all corners are marked as non-ideal. 

For these top-dimension polytopes, we can introduce the following natural volume measure given by $2^{n-3}$ times the $(2n-6)$-dimensional Lebesgue measure
\begin{align}
	\operatorname{Vol}_\tree = 2^{n-3} \prod_{\vertex\in \vsetinner(\tree)}\rmd\varphi(\vec{\edge}_{\vertex,1})\rmd\varphi(\vec{\edge}_{\vertex,2}) \prod_{\bvertex\in \vsetboundary(\tree)} \prod_{j=1}^{\deg(\bvertex)-1}\rmd w_{\bvertex,j}\rmd v_{\bvertex,j},\label{eq:measurepolytope}
\end{align}
where $(w_{\bvertex,1},\ldots,w_{\bvertex,\deg(\bvertex)},v_{\bvertex,1},\ldots,v_{\bvertex,\deg(\bvertex)})$ parametrizes the simplices of $\bvertex$.

We can now formulate the first bijection:
\begin{theorem}\label{thm:HTC_bij}
    There exists a bijection
    \begin{align}
        \spinebijection:\mathcal{H}_n(\mathbf{L})\to \bigsqcup_{\tree \in \treeset^{\mathrm{all,HTC}}_n} \mathcal{A}_{\tree}(\mathbf{L}).
    \end{align} 
    Moreover, the push-forward of the Weil--Petersson measure is the measure $\operatorname{Vol}_\tree$ on the top-dimensional polytopes.
\end{theorem}

\begin{theorem}\label{thm:HTC_gf}
    Let the generating function $R[\mu]$ be the unique formal power series solution satisfying $R[\mu] = \int\rmd\mu(L) + O(\mu^2)$ to
    \begin{equation}\label{eq:Zdef}
    Z(R[\mu];\mu] = 0,\qquad Z(r;\mu]\coloneqq \frac{\sqrt{r}}{\sqrt{2}\pi}J_1(2\pi\sqrt{2r}) -  \int\rmd\mu(L)\,I_0(L\sqrt{2r})
    \end{equation}
    where $I_0$ and $J_1$ are (modified) Bessel functions. 
    
    The generating function of half-tight cylinders is given by 
    \begin{align}
        H(L_1,L_2;\mu] = \sum_{k=0}^\infty \frac{2^{-k} R[\mu]^{k+1}}{k!(k+1)!}(L_2^2-L_1^2)^{k}.
    \end{align}
\end{theorem}
    Note that this result has already been proven in \cite{Budd2024_Top_rec}, but this is done by solving algebraic equations, instead of a bijective proof like done here.

\subsection{General planar surfaces}
We now introduce the set of bicolored \emph{double} plane trees $\dubtree=(\tree_1,\tree_2)\in\treeset^{\mathrm{all,double}}_n$, each with identical restrictions as before, except that there are in total $n$ boundary vertices with labels $1,\ldots,n$, partitioned such that $b_1\in\vsetboundary(\tree_1)$ and $b_2\in\vsetboundary(\tree_2)$ and each tree has at least two boundary vertices.

To a $\dubtree\in\treeset^{\mathrm{all,double}}_n$ we associate the \emph{full polytope} $\mathcal{A}_{\dubtree}(\mathbf{L})$: 
\begin{align}
	\mathcal{A}_{\dubtree}(\mathbf{L}) &=
    \left\{
        \left(\ell,\tau,
        (w_{\bvertex_1,i})_{i=1}^{\deg(\bvertex_1)},
        (v_{\bvertex_1,i})_{i=1}^{\operatorname{nonid}(\bvertex_1)},
        (w_{\bvertex_2,i})_{i=1}^{\deg(\bvertex_2)},
        (v_{\bvertex_2,i})_{i=1}^{\operatorname{nonid}(\bvertex_2)}\right) 
        \in \R^{d_{12}}
    \middle| 
    \begin{array}{l}
        0\leq\tau<\ell<\min(L_1,L_2)\\
        \mathbf{w}_{\bvertex_i}\in \Delta_{\deg(\bvertex_i)}^{((L_i-\ell)/2)}\\
        \mathbf{v}_{\bvertex_1}\in \Delta_{\operatorname{nonid}(\bvertex_i)}^{((L_i+\ell)/2)} 
    \end{array}\right\}
    \nonumber\\&\qquad
    \times \left\{ \varphi : \vec{\edgeset}(\dubtree) \to \R_{\geq 0} \middle|
    \begin{array}{l}
        \varphi(\vec{\edge}_{\bvertex,i}) = 0\text{ for }\bvertex\in\vsetboundary(\dubtree), 1\leq i\leq\deg(\bvertex),\\
        \varphi(\vec{\edge}_{\vertex,i}) > 0\text{ for }\vertex\in \vsetinner(\dubtree),  1\leq i\leq\deg(\vertex),\\
        \sum_{i=1}^{\deg(\vertex)}\varphi(\vec{\edge}_{\vertex,i}) = \pi\text{ for }\vertex\in \vsetinner(\dubtree), \\ 
        \varphi(\vec{\edge}) + \varphi(\cev{\edge}) < \pi\text{ for }\edge\in \edgeset(\dubtree)
    \end{array}\right\} \nonumber\\&\qquad
    \times \prod_{\bvertex\in \vsetboundary(\tree)\setminus\{\bvertex_1,\bvertex_2\}} \left[\Delta_{\deg(\bvertex)}^{(L_{\bvertex}/2)} \times \Delta_{\operatorname{nonid}(\bvertex)}^{(L_{\bvertex}/2)}\right],
    \label{eq:Full_polytopedef}
\end{align}
where $d_{12}=2+\deg(\bvertex_1)+\operatorname{nonid}(\bvertex_1)+\deg(\bvertex_2)+\operatorname{nonid}(\bvertex_2)$.

Furthermore, the top-dimensional polytopes again correspond to the trees 
where all inner vertices are of degree $3$ and all corners are marked as non-ideal. 
On the top-dimensional full polytopes, we have a similar volume measure given by $2^{n-4}$ times the $(2n-6)$-dimensional Lebesgue measure
\begin{align}
	\operatorname{Vol}_\tree = 2^{n-4}\,\rmd\ell\rmd\tau \prod_{\vertex\in \vsetinner(\tree)}\rmd\varphi(\vec{\edge}_{\vertex,1})\rmd\varphi(\vec{\edge}_{\vertex,2}) \prod_{\bvertex\in \vsetboundary(\tree)} \prod_{j=1}^{\deg(\bvertex)-1}\rmd w_{\bvertex,j}\rmd v_{\bvertex,j}.\label{eq:measure_full_polytope}
\end{align}

We get the following result.
\begin{theorem}\label{thm:Full_bij}
    Given $\mathbf{L}=(L_1,\dots,L_n)$ with $L_1< L_2$, let $\mathcal{M}^\circ_{0,n}(\mathbf{L})\subset\mathcal{M}_{0,n}(\mathbf{L})$ be the full WP-measure subset, where there is a unique shortest geodesic separating the first two boundaries, possibly being one of the boundaries itself.
    Then there is a bijection
    \begin{align}
        \spinebijection:\mathcal{M}^\circ_{0,n}(\mathbf{L})\to \left(\bigsqcup_{\tree \in \treeset^{\mathrm{all,HTC}}_n} \mathcal{A}_{\tree}(\mathbf{L})\right)\sqcup\left(\bigsqcup_{\dubtree \in \treeset^{\mathrm{all,full}}_n} \mathcal{A}_{\dubtree}(\mathbf{L})\right).
    \end{align}
    Moreover, the push-forward of the Weil--Petersson measure is the measure $\operatorname{Vol}_\tree$ on the top-dimensional half-tight and full polytopes.
\end{theorem}

This theorem allows us to compute the Weil--Petersson volumes of planar hyperbolic surfaces, using a recursion in the number of boundaries:
\begin{theorem} \label{thm:recursion}
    Let $(f_3,f_4,\cdots)$ be the sequence of polynomials 
    \begin{align}
        f_n(\hat{t}_0,\cdots,\hat{t}_{n-3};1/\hat{\gamma}_1,\hat{\gamma}_2,\cdots,\hat{\gamma}_{n-2})\in \Z[\hat{t}_0,\cdots,\hat{t}_{n-3},1/\hat{\gamma}_1,\hat{\gamma}_2,\cdots,\hat{\gamma}_{n-2}],
    \end{align}
    defined recursively by 
    \begin{align}
        f_3(\hat{t}_0;1/\hat{\gamma}_1)=-\frac{\hat{t}_0^3}{\hat{\gamma}_1}
    \end{align} 
    and for $n\geq3$
    \begin{align}\label{eq:f_rec}
        &f_{n+1}(\hat{t}_0,\cdots,\hat{t}_{n-2};1/\hat{\gamma}_1,\hat{\gamma}_2,\cdots,\hat{\gamma}_{n-1})=\nonumber\\&\qquad\qquad
        \sum_{k=0}^{n-3}\left[
            \hat{t}_{k+1}\left(\pdv{\hat{\gamma}_{k+1}}-\frac{\hat{t}_0}{\hat{\gamma}_1}\pdv{\hat{t}_k}\right) 
            -\hat{\gamma}_{k+2}\frac{\hat{t}_0}{\hat{\gamma}_1}\pdv{\hat{\gamma}_{k+1}}
        \right] f_n(\hat{t}_0,\cdots,\hat{t}_{n-3};1/\hat{\gamma}_1,\hat{\gamma}_2,\cdots,\hat{\gamma}_{n-2}).
    \end{align}
\\\\
    Then the $n$-th term of the generating function of genus-$0$ Weil--Petersson volumes is given by
    \begin{align}
        n![x^n]F_0(x\mu)=\int\rmd\mu(L_1)\cdots\rmd\mu(L_n)V_{0,n}(L_1,\cdots,L_n)=
        \frac{1}{8} f_n(t_0[\mu],\cdots,t_{n-3}[\mu];1/\gamma_1,\gamma_2,\cdots,\gamma_{n-2}),
    \end{align}
    where
    \begin{align}
        t_k[\mu]=\int \rmd\mu(L)\, 2\frac{L^{2k}}{4^k k!}
    \end{align}
    and 
    \begin{align}
        \gamma_k=(-1)^k \frac{\pi^{2k-2}}{(k-1)!}.
    \end{align}
\end{theorem}
In fact, this theorem is a special case of \cite[Theorem~4]{Budd2024_Top_rec},%
\footnote{Which is itself equivalent to \cite[Proposition~3.5]{Bouttier2024_Enumeration_BGM_D}, which is a similar recursion for tight maps.} 
although this is hard to recognize, due to different notation. 
In any case, in \cite{Budd2024_Top_rec} this theorem is proven by solving algebraic equations, while here we will provide a more direct proof. 

\subsection{Outline}
After some short background on moduli spaces and Weil--Petersson volumes in section~\ref{sec:background}, we start generalizing the tree bijection of \cite{Budd2025_Tree_bij}. 
In section~\ref{sec:spine}, we use Busemann functions to construct the spine for a half-tight cylinder. 
Section~\ref{sec:tiling} completes the proof of the bijection for the half-tight cylinder.

Next, in section~\ref{sec:general}, we extend the tree bijection to general surfaces, by decomposing the surfaces into half-tight cylinders.
We use these results in section~\ref{sec:compute} where we use some tricks to simplify the computation of Weil--Petersson volumes. 
Finally, in section~\ref{sec:outlook}, we give some outlook.

\subsection{Acknowledgments}
This work is supported by the VIDI programme with project number VI.Vidi.193.048, which is financed by the Dutch Research Council (NWO).
Furthermore, the author wants to thank Timothy Budd for the fruitful cooperations on this topic.

\section{Background}\label{sec:background}
Let's start with recalling some definitions. 
Given a topological surface $S_{g,n}$ of genus $g$ and with $n$ punctures, we can look at the corresponding Teichm\"uller space $\mathcal{T}_{g,n}(L_1,\cdots,L_n)$, which can be defined as all hyperbolic metrics on $S_{g,n}$, such that the $n$ boundary curves are geodesics with lengths $L_1,\ldots,L_n$, modulo boundary- and orientation-preserving  isometries that are homotopic to the identity.
Alternatively, one can also mod out all boundary- and orientation-preserving  isometries, which gives the moduli space $\mathcal{M}_{g,n}(L_1,\cdots,L_n)=\mathcal{T}_{g,n}(L_1,\cdots,L_n)/\textrm{MCG}_{g,n}$, where $\textrm{MCG}_{g,n}$ is the so-called mapping class group.

The Teichm\"uller space comes equipped with a natural symplectic form $\omega_\textrm{WP}$, called the Weil--Petersson form.
There are multiple ways to define $\omega_\textrm{WP}$, but commonly it is computed using a pants decomposition and Fenchel--Nielsen coordinates.

A pants decomposition is a collection of disjoint simple closed curves $\Gamma=(\gamma_i)_{i=1}^{3g-3+n}$ on $S_{g,n}$, such that the complement only consists of pairs of pants, which are homeomorphic to $S_{0,3}$. 

It is shown that, given a pants decomposition $\Gamma$, the real numbers $(\ell_{\gamma_i}(\cdot),\tau_{\gamma_i}(\cdot))_{i=1}^{3g-3+n} \in (\R_+ \times \R)^{3g-3+n}$ are coordinates on $\mathcal{T}_{g,n}(L_1,\cdots,L_n)$, where $\ell_\gamma(X)$ is the length of the unique closed geodesic in $X$ homotopic to $\gamma$ and $\tau_\gamma(X)$ are the `twists' that one needs to apply when gluing the pairs of pants along $\gamma$.

Still with a fixed pants decomposition $\Gamma$, the Weil--Petersson form $\omega_\textrm{WP}$ is now easily expressed by \cite{Wolpert1983_symplectic}
\begin{align}
    \omega_\textrm{WP}=\sum_{\gamma\in\Gamma} \dd{\ell_\gamma}\wedge\dd{\tau_\gamma}.
\end{align}
Perhaps surprisingly, the result is independent on the choice of $\Gamma$ and is even invariant under the action of the mapping class group, which means it descends to the moduli space $\mathcal{M}_{g,n}(L_1,\cdots,L_n)$.
Since this space is $6g-6+2n$ dimensional, we get a natural volume measure on $\mathcal{M}_{g,n}(L_1,\cdots,L_n)$ given by
\begin{align}
    \mu_{\textrm{WP}}=\frac{\omega^{3g-3+n}}{(3g-3+n)!},
\end{align}
which is simply the Lebesgue measure on $(\ell_i,\tau_i)_{i=1}^{3g-3+n} \in (\R_+ \times \R)^{3g-3+n}$, where we used shorthand notation $\ell_i=\ell_{\gamma_i}(\cdot)$.
We call this measure the Weil--Petersson measure and we call the total volumes
\begin{align}
    V_{g,n}(L_1,\cdots,L_n)=\int_{\mathcal{M}_{g,n}(L_1,\cdots,L_n)}\mu_{\textrm{WP}}
\end{align}
the Weil-Petersson volumes. 

In \cite{Mirzakhani2006_SimpleGeodesics_TopRec}, Mirzakhani proves that they are rational homogeneous polynomials in $\pi^2$ and $L_i^2$, and she proves a recursion relation that allows for the computation of any Weil--Petersson volume.

\pagebreak[0]
The first few Weil--Petersson volumes are given by \cite{Do2011_ModuliSpaces_WPTable}

\nopagebreak
\begin{tabular}{lll}
    $g$&$n$&$V_{g,n}$\\
    \hline
    $0$&$3$&$1$\\
    $0$&$4$&$2\pi^2+\frac{1}{2}\sum_i L_i^2$\\
    $0$&$5$&$10\pi^4+3\pi\sum_i L_i^2+\frac{1}{8}\sum_i L_i^4+\frac{1}{2}\sum_{i<j}L_i^2L_j^2$\\
    $0$&$6$&$\frac{244}{3}\pi^6+26\pi^4\sum_i L_i^2+
    \frac{3}{2}\pi^2\sum_i L_i^4+6\pi^2\sum_{i<j}L_i^2L_j^2+
    \frac{1}{48}\sum_i L_i^6$\\ &&$\qquad+\frac{3}{16}\sum_{i\neq j} L_i^4L_j^2+\frac{3}{4}\sum_{i\neq j \neq k \neq i}L_i^2L_j^2L_k^2$\\
    \hline
    $1$&$1$&$\frac{1}{12}\pi^2+\frac{1}{48}L_i^2$\\
    $1$&$2$&$\frac{1}{4}\pi^4+\frac{1}{12}\pi^2\sum_iL_i^2+\frac{1}{192}\sum_iL_i^4+\frac{1}{96}L_1^2L_2^2$\\
    $1$&$3$&$\frac{14}{9}\pi^6+\frac{13}{24}\pi^4\sum_iL_i^2+\frac{1}{24}\pi^2\sum_iL_i^4+\frac{1}{8}\pi^2\sum_{i<j}L_i^2L_j^2+\frac{1}{1152}\sum_iL_i^6$\\ &&$\qquad+\frac{1}{192}\sum_{i\neq j} L_i^4L_j^2+\frac{1}{96}L_1^2L_2^2L_3^2$\\
    \hline
    $2$&$0$&$\frac{43}{2160}\pi^6$\\
    $2$&$1$&$\frac{29}{192}\pi^8+\frac{168}{2880}\pi^6L_1^2+\frac{139}{23040}\pi^4L_1^4+\frac{29}{138240}\pi^2L_1^6+\frac{1}{442367}L_1^8$
\end{tabular}

\nopagebreak
\vspace{1 em} In this paper, we will only focus on the planar case ($g=0$). 

\section{The spine on a half-tight cylinder} \label{sec:spine}
We are now ready to extend the arguments from \cite{Budd2025_Tree_bij} to a half-tight cylinder.

\subsection{The cylinder covering map}
\begin{figure}
    \centering
    \begin{subfigure}[t]{.3\textwidth}
        \centering
        \includegraphics[width=\textwidth]{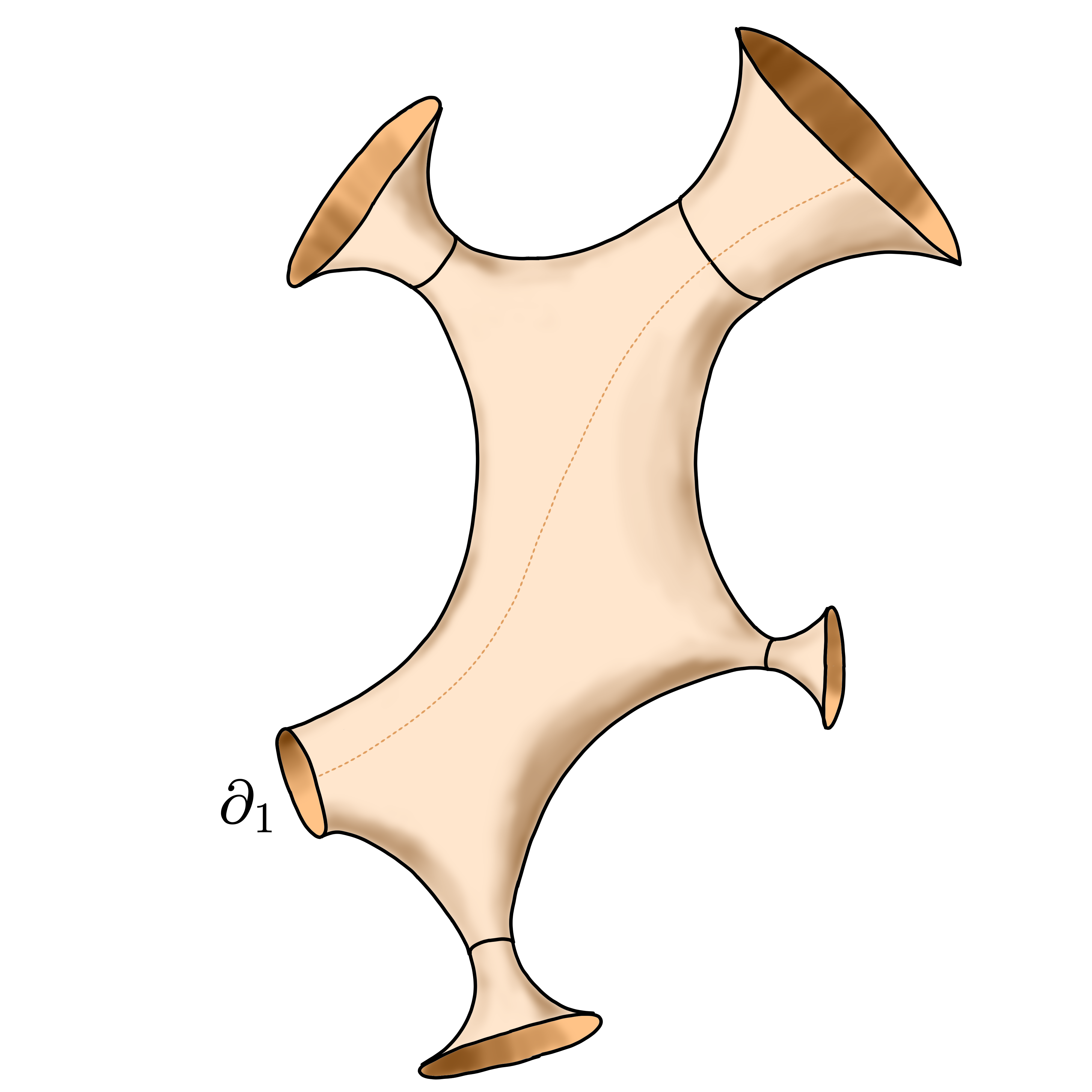}
        \caption{The extended surface $\check{X}$.}
    \end{subfigure}
    \begin{subfigure}[t]{.65\textwidth}
        \centering
        \includegraphics[width=\textwidth]{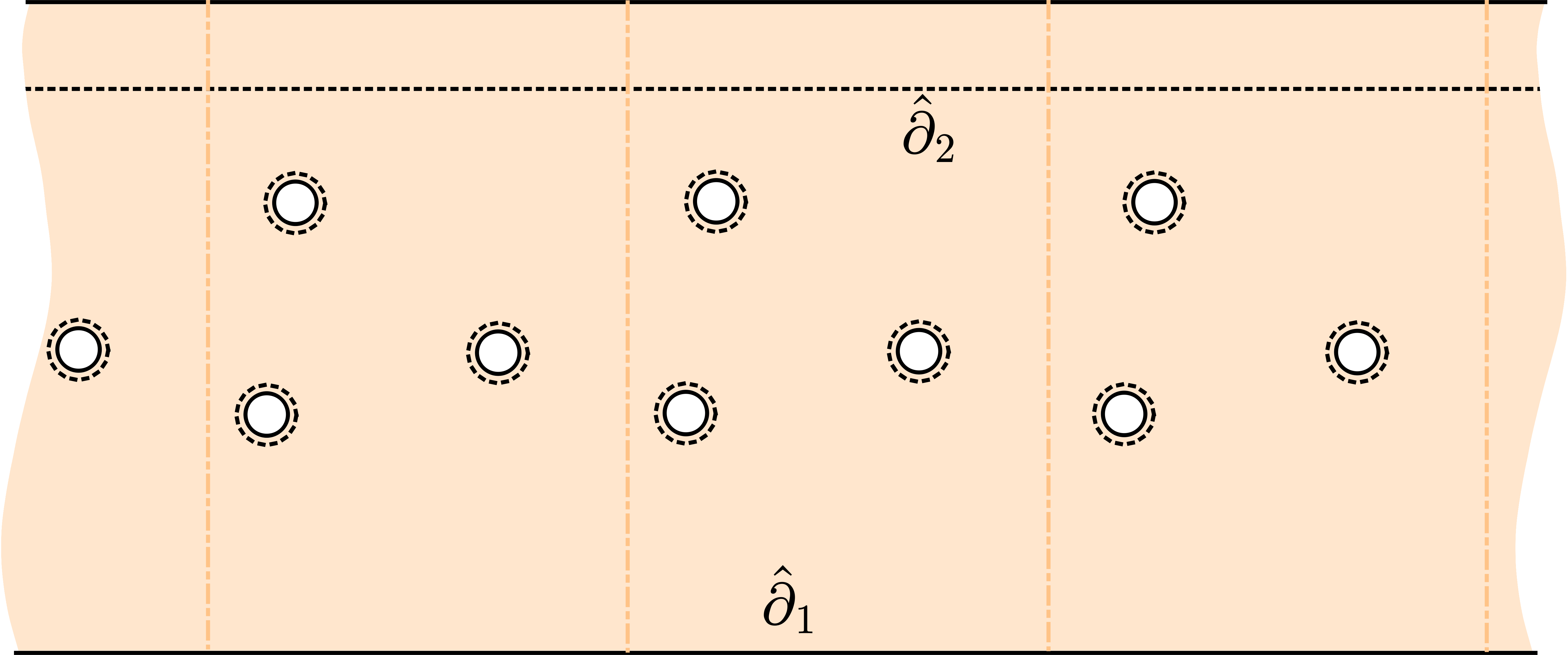}
        \caption{The cylinder cover $\hat{X}$. The black line below corresponds to $\partial_1$. The black dotted lines are the original boundaries in $X$, where the infinite dotted line is $\partial_2$. The region between two yellow lines is a fundamental domain.
        \\
        Note that this is only a qualitative depiction of the cylinder cover: geodesics do not necessarily correspond to straight lines.
        }
    \end{subfigure}
    \caption{}
\end{figure}

Consider a half-tight cylinder $X\in \mathcal{H}_n(\mathbf{L})$, with $\mathbf{L}\in\R_{>0}^n$. 
We extend $X$ to $\check{X}$ by gluing infinite trumpets to all boundaries, except to $\partial_1$.
We consider the cylinder covering space $\hat{X}$ with corresponding covering map $p:\hat{X}\rightarrow \check{X}$, which is the unique covering space, such that $\partial_1$ and $\partial_2$ are lifted to unique infinite geodesics $\hat{\partial}_1$ and $\hat{\partial}_2$ respectively, while all other boundary curves have an infinite number of lifts, each of the same length as in $X$.

Note that the group of deck transformations is isomorphic to $\Z$.
We call the generator of this group $A$.
Given a length-parametrization $\gamma(t)$ of $\hat{\partial}_1$, such that for increasing $t$ the inside of $\hat{X}$ is on the right-hand side, the action of this generator is given by $A\gamma(t)=\gamma(t+L_1)$.

We let $\hat{X}$ inherit the metric from $\check{X}$, and we denote its distance function by $d(\cdot,\cdot)$.

\begin{lemma}\label{lem:minimizing_geod}
    The boundary $\hat{\partial}_1$ of $\hat{X}$ is a minimizing geodesic.
    Specifically, let $\gamma(t)$ for $t\in\R$ be a length-parametrization of $\hat{\partial}_1$, such that for increasing $t$ the inside of $\hat{X}$ is on the right-hand side.
    Then 
    \begin{align}
        d(\gamma(t_1),\gamma(t_2))=|t_1-t_2| \textrm{ for all } t_1,t_2\in \R.
    \end{align}
\end{lemma}
In terminology of \cite{papadopoulos2014metric}, $\gamma$ is a geodesic ray in $\hat{X}$.
\begin{proof}
    We will prove this by contradiction. 
    We start with the assumption that there exists a pair $t_1$, $t_2$, such that 
    \begin{align}
        d(\gamma(t_1),\gamma(t_2))<|t_1-t_2|.
    \end{align}
    Without loss of generality, we assume $t_1<t_2$.
    This means that there is a geodesic with length-parametrization $\tilde{\gamma}: [0,T]\rightarrow \hat{X}$, with $\tilde{\gamma}(0)=\gamma(t_1)$ and $\tilde{\gamma}(T)=\gamma(t_2)$ and $T<t_2-t_1$.
    
    \begin{figure}
        \centering
            \centering
            \includegraphics[width=.7\textwidth]{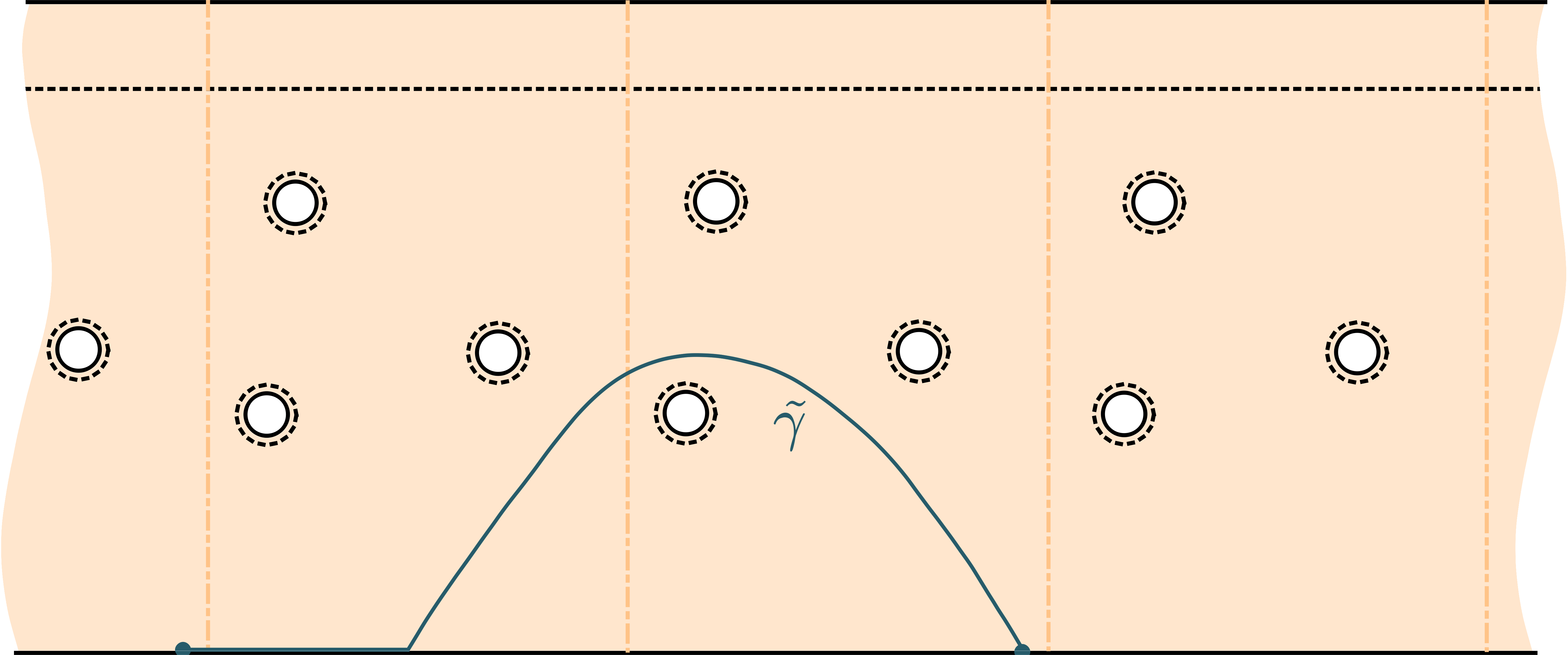}
            \caption{See Lemma~\ref{lem:minimizing_geod}. The blue curve is $\tilde{\gamma}$ concatenated with $\left.\gamma\right|_{[t_2,t_1+2 L_1]}$, which wraps twice around the cylinder.}
    \end{figure}

    We have $(n-1) L_1<t_2-t_1\leq n L_1$ for some $n=1,2,3,\cdots$. 
    We concatenate $\tilde{\gamma}$ with $\left.\gamma\right|_{[t_2,t_1+n L_1]}$. 
    The image of this concatenation under $p$ is a closed separating curve, which wraps $n$ times around the cylinder.
    It can be naturally decomposed into $n$ closed separating curves. 
    The total length is $T+t_1+n L_1-t_2<n L_1$, which means that at least one of the closed separating curves is shorter than $L_1$.
    This contradicts tightness of $\partial_1$, so $\hat{\partial_1}$ is minimizing.
\end{proof}
This is a crucial lemma, since it allows us to define a Busemann function on $\hat{X}$.
It is a well-known function in metric geometry (including hyperbolic geometry, but also for maps). 
See for example \cite{papadopoulos2014metric} for more information on Busemann functions and related objects.

\subsection{The Busemann function and asymptotically parallel geodesics}
We consider any parametrization $\gamma(t)$ of $\hat{\partial}_1$ from Lemma~\ref{lem:minimizing_geod}, such that for increasing $t$ the inside of $\hat{X}$ is on the right-hand side.
We can now define its \emph{Busemann function} $B_\gamma: \hat{X}\rightarrow \R$
\begin{align}
    B_\gamma(z)=\lim_{t\rightarrow\infty}\left(d(z,\gamma(t))-t\right).
\end{align}
Note that this limit is well-defined, since the argument is decreasing in $t$ and bounded from below due to the triangle inequality.
The Busemann function is clearly 1-Lipschitz.

The Busemann function depends on the choice of the base point $\gamma(0)\in\hat{\partial}_1$. 
In particular, when we choose $\gamma(t_0)$ as a new base-point, the Busemann function increases by $t_0$.
As a direct consequence, for different lifts of a point $z\in \check{X}$, the values of $B_\gamma(\hat{z})$ for $\hat{z}\in p^{-1}(z)$ differ by multiples of $L_1$, specifically\footnote{Assuming the correct orientation of $A$.} 
\begin{align}\label{eq:bus_deck}
    B_\gamma(Az)=B_\gamma(z) - L_1.
\end{align} 
This means that we cannot project the Busemann function to $\check{X}$, but locally in $\check{X}$, we can have a well-defined Busemann function, up to shifts of $L_1$.

We can consider \emph{asymptotically parallel geodesics} $\alpha_z: [0,\infty)\rightarrow\hat{X}$ to $\gamma$ starting at $z$. 
They can be defined as a limit of geodesics from $z$ to increasingly further points on $\gamma$, see \cite{papadopoulos2014metric}. 
We can define its (signed) length as
\begin{align}
    \ell_\gamma(\alpha_z)=\lim_{t\to\infty}(t+B_\gamma(\alpha(t))),
\end{align} 
which might be $+\infty$.

Since $\gamma$ is a geodesic ray, for each point $z\in\hat{X}$, there exists at least one co-ray $\zeta_z: [0,\infty)\rightarrow\hat{X}$, which is an asymptotically parallel geodesic to $\gamma$ starting at $z$ with minimal length $\ell_\gamma(\zeta_z)=B_\gamma(z)$. 
It is known that along its path the Busemann function with respect to $\gamma$ decreases with rate $1$, so for any $t_1,t_2\in[0,\infty)$, we have
\begin{align}\label{eq:co-ray_bus}
    B_\gamma(\zeta_z(t_1))-B_\gamma(\zeta_z(t_2))=t_2-t_1.
\end{align} 
See \cite{papadopoulos2014metric}.

\begin{lemma}
   Let $\zeta_{z_1}$ and $\zeta'_{z_2}$ be two co-rays to $\gamma$, and $p\circ\zeta_{z_1}$ and $p\circ\zeta'_{z_2}$ their projections to $\check{X}$.
   Assume that the co-rays are chosen such that the projections are not sub-curves of each other and such that the images in $\check{X}$ are not simple closed separating geodesics.
   
   Then, the projections $p\circ\zeta_{z_1}$ and $p\circ\zeta'_{z_2}$ are geodesics, starting at $p(z_1)$ and $p(z_2)$ respectively.
   Moreover, they do not have (self-)intersections in $\check{X}$, except for possibly their starting points.
\end{lemma}
\begin{proof}
    Firstly, since $\zeta_{z_1}$ and $\zeta'_{z_2}$ are geodesics, so are $p\circ\zeta_{z_1}$ and $p\circ\zeta'_{z_2}$.

    Furthermore, we note that it is sufficient to show that $\zeta_{z_1}$ and $\zeta'_{z_2}$ don't (self-)intersect, since if $p\circ\zeta_{z_1}$ and $p\circ\zeta'_{z_2}$ (self)-intersect, there exists lifts that (self-)intersect.

    Now a co-ray $\zeta_{z_1}$ cannot self-intersect, since it would violate equation \eqref{eq:co-ray_bus}.
    
    Now for intersections between $\zeta_{z_1}$ and $\zeta'_{z_2}$.
    Say $\zeta_{z_1}(t_1)=\zeta'_{z_2}(t_2)$ for some $t_1,t_2>0$. 
    For sufficient small $\epsilon$, we have 
    \begin{align}
        d\left(\zeta_{z_1}(t_1-\epsilon),\zeta'_{z_2}(t_2+\epsilon)\right)<2\epsilon,
    \end{align}
    since the intersection is transverse. On the other hand, using \eqref{eq:co-ray_bus} twice, we get
    \begin{align}
        B_\gamma\left(\zeta_{z_1}(t_1-\epsilon)\right)-B_\gamma\left(\zeta'_{z_2}(t_2+\epsilon)\right)=2\epsilon.
    \end{align}
    Together, this violates the fact that the Busemann function is 1-Lipschitz, so the co-rays cannot intersect.
\end{proof}
\begin{lemma}\label{lem:spiral_to_partial_1}
    Let $\zeta_{z_1}$ be a co-ray to $\gamma$, and $p\circ\zeta_{z_1}$ its projection to $\check{X}$.
    Assume that the co-ray is chosen such that the image in $\check{X}$ is not equal to $\partial_1$.
    
    Then, the projection $p\circ\zeta_{z_1}$ spirals towards $\partial_1$.
\end{lemma}
\begin{proof}
    First note that, due to the previous lemma, the projection $p\circ\zeta_z$ is either a simple closed geodesic or a non-self-intersecting infinite geodesic. 
    
    In the first case, going around the closed curve once, reduces the Busemann function by $L_1$ according to \eqref{eq:bus_deck}, which means that the length of the curve is also $L_1$, due to \eqref{eq:co-ray_bus}. 
    Since $\partial_1$ is strictly tight, it is the only separating closed curve with this length. 
    
    \begin{figure}
        \centering
            \centering
            \includegraphics[width=.7\textwidth]{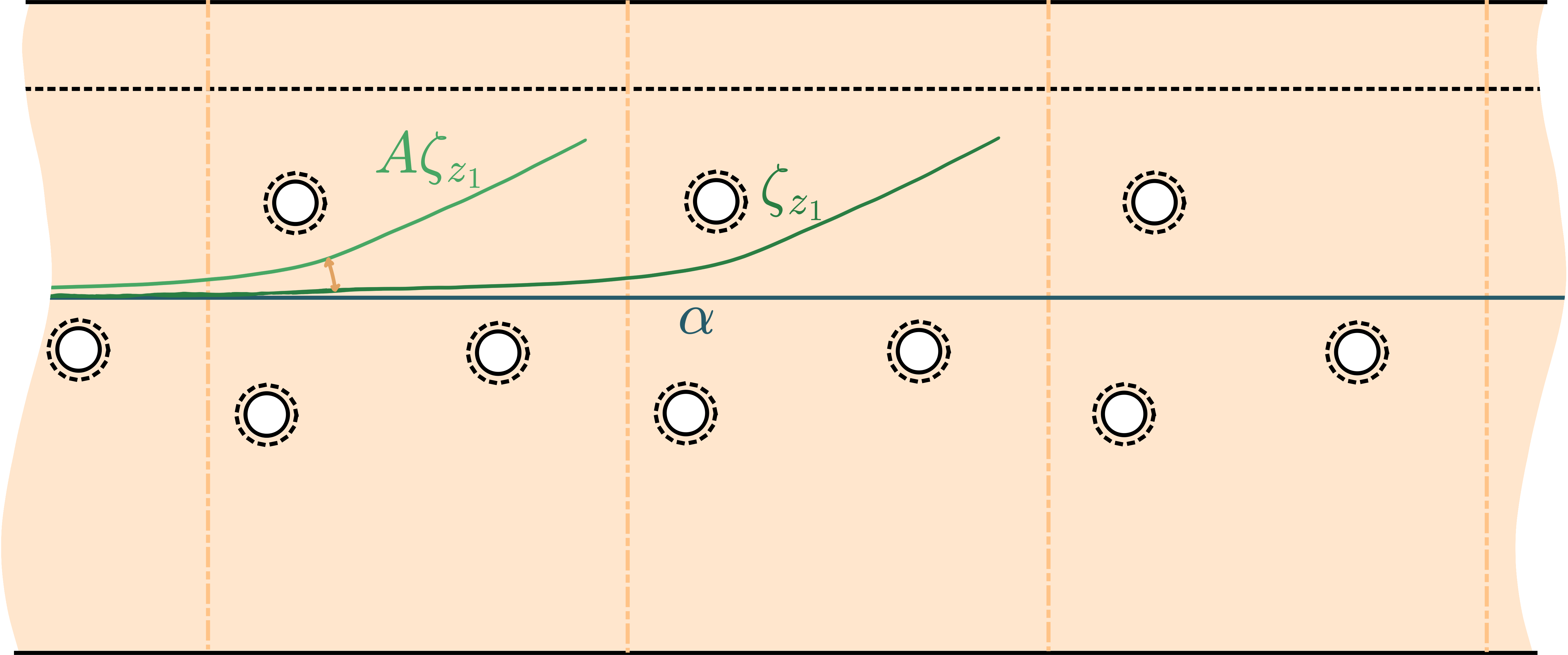}
            \caption{See Lemma~\ref{lem:spiral_to_partial_1}. We have two geodesics $\zeta_{z_1}$ and $A\zeta_{z_1}$ in green, spiraling towards a curve $\alpha$ in blue. Since they spiral to the same curve, we have for some $t'$ that $d\left(A\zeta_z(t'),\zeta_z(t'+\ell_\alpha)\right)<\delta$.}
    \end{figure}

    In the second case the projection $p\circ\zeta_z$, being a non-self-intersecting infinite geodesic, will spiral to some simple closed geodesic $\alpha$.%
    \footnote{We use the argumentation in \cite[Paragraph after Theorem~4.5]{Mirzakhani2006_SimpleGeodesics_TopRec} for this claim.} 
    It is clear that this geodesic must be separating, since the Busemann function of $\zeta_z$ is decreasing.
    This means that, unless $p\circ\zeta_z$ spirals towards $\partial_1$, the limiting closed geodesic $\alpha$ has length larger than $L_1$. 
    Moreover, due to the fact that the number of simple closed geodesics with bounded lengths is finite, we can find a $\delta>0$ such that $\ell_\alpha>L_1+\delta$. 
    The fact that $p\circ\zeta_z$ spirals towards $\alpha$ means that we can find a $t'$ such that 
    \begin{align}
        d\left(p\circ\zeta_z(t'),p\circ\zeta_z(t'+\ell_\alpha)\right)<\delta,
    \end{align}
    which means in the cover $\hat{X}$ we have
    \begin{align}
        d\left(A\zeta_z(t'),\zeta_z(t'+\ell_\alpha)\right)<\delta.
    \end{align}
    Meanwhile,
    \begin{align}
       B_\gamma(A\zeta_z(t'))-B_\gamma(\zeta_z(t'+\ell_\alpha))=B_\gamma(\zeta_z(t'))-L_1-\left(B_\gamma(\zeta_z(t'))-\ell_\alpha\right)=\ell_\alpha-L_1>\delta.
    \end{align}
    This is in contradiction with the fact that Busemann functions are 1-Lipschitz.
\end{proof}

We call these projections of co-rays to $\gamma$ \emph{base spirals}.
Note that the base spirals are independent of the choice of base point on $\gamma$.
For any point $x\in\check{X}$, we denote by $w(x)$ the number of base spirals starting at $x$. 

To get a sense of local structure it is also useful to look at asymptotically parallel geodesics that have almost minimal length. 
We call an asymptotically parallel geodesic in $\alpha_{\hat{x}}\subset\hat{X}$ starting at $\hat{x}$ \emph{$\epsilon$-optimal} for $\epsilon>0$ when its length is bounded by $\ell_\gamma(\alpha_{\hat{x}})<B_\gamma(\hat{x})+\epsilon$ and its projection to $\check{X}$ has no self intersections.

For our starting point $x\in\check{X}$ we pick a shortest geodesic $\beta_{*}$ from $x$ to the infinite boundary of $\partial_2$, consider its lifts $\hat{\beta}_{*,j}$, where $j\in\Z$ indexes the lifts, such that $\hat{\beta}_{*,0}$ starts at $\hat{x}$ and $A\hat{\beta}_{*,j}=\hat{\beta}_{*,j+1}$. 
Furthermore, we pick shortest geodesics $\beta_{i}\subset\check{X}$ from the infinite boundary of $\partial_i$ to the infinite boundary of $\partial_2$ for $i\in\{3,\dots,n\}$ and also consider their lifts $\hat{\beta}_{i,j}$ in $\hat{X}$, such that $\hat{\beta}_{i,j}$ end on the infinite boundary of $\partial_2$ between $\hat{\beta}_{*,j}$ and $\hat{\beta}_{*,j+1}$. 
Observe that $\hat{X}\setminus\bigcup_{i,j}\hat{\beta}_{i,j}$ is simply connected.

\begin{lemma}\label{lem:eps_opt}
Let $\alpha_{\hat{x}}$ be an $\epsilon$-optimal asymptotically parallel geodesic for some $\epsilon>0$. 
Then
\begin{enumerate}
    \item\label{enum:eps_opt_finite} $\alpha_{\hat{x}}$ intersects each $\hat{\beta}_{i,j}$ finitely many times.
    \item\label{enum:eps_opt_K} If $\alpha_{\hat{x}}$ intersects $\hat{\beta}_{i,j}$ for some $i,j$, there exists a $K>0$ such that $\alpha_{\hat{x}}$ cannot intersect $\hat{\beta}_{i,j+k}$  for $k>K$.
    \item For all $i$ there exists a $J_{\mathrm{min}}<0$ such that $\alpha_{\hat{x}}$ cannot intersect $\hat{\beta}_{i,j}$ for $j<J_{\mathrm{min}}$.
    \item\label{enum:eps_opt_j_above} For all $i$ there exists a $J_{\mathrm{max}}>0$ such that $\alpha_{\hat{x}}$ cannot intersect $\hat{\beta}_{i,j}$ for $j>J_{\mathrm{max}}$.
\end{enumerate}
In particular, there are finitely many homotopy classes that contain an $\epsilon$-optimal asymptotically parallel geodesic.
\end{lemma}
\begin{corollary}\label{cor:finite_w}
    For each $x\in\check{X}$, the number of (projections of) $\epsilon$-optimal asymptotically parallel geodesics is finite. 
    In particular, $w(x)$ is finite.
\end{corollary}
\begin{proof}[Proof of Lemma~\ref{lem:eps_opt}]
    Note that the fact that $\alpha_{\hat{x}}$ is $\epsilon$-optimal implies for any $0<t_1<t_2$
    \begin{align}\label{eq:eps_opt_Busemann}
        B_\gamma(\alpha_{\hat{x}}(t_1))-B_\gamma(\alpha_{\hat{x}}(t_2))>t_2-t_1-\epsilon.
    \end{align}
    Note that the Busemann function on $\alpha_{\hat{x}}$ is thus bounded from above by $B_\gamma(\alpha_{\hat{x}}(0))+\epsilon$.

    On the other hand, for each $i,j$ the Busemann function on $\hat{\beta}_{i,j}$ is bounded from below.
    In particular, we can use the fact that $B_\gamma(Ay)=B_\gamma(y)-L_1$ for any $y\in\hat{X}$, to see that there is a constant $B_0$ such that the Busemann function on $\hat{\beta}_{i,j}$ is larger than $B_0-j L_1$.

    Combining these facts, we get that for each $i,j$ the Busemann functions on intersections of $\alpha_{\hat{x}}$ and $\hat{\beta}_{i,j}$ are bounded, where the upper bound does not depend on $i,j$. 

    \begin{figure}
        \centering
            \centering
            \includegraphics[width=.7\textwidth]{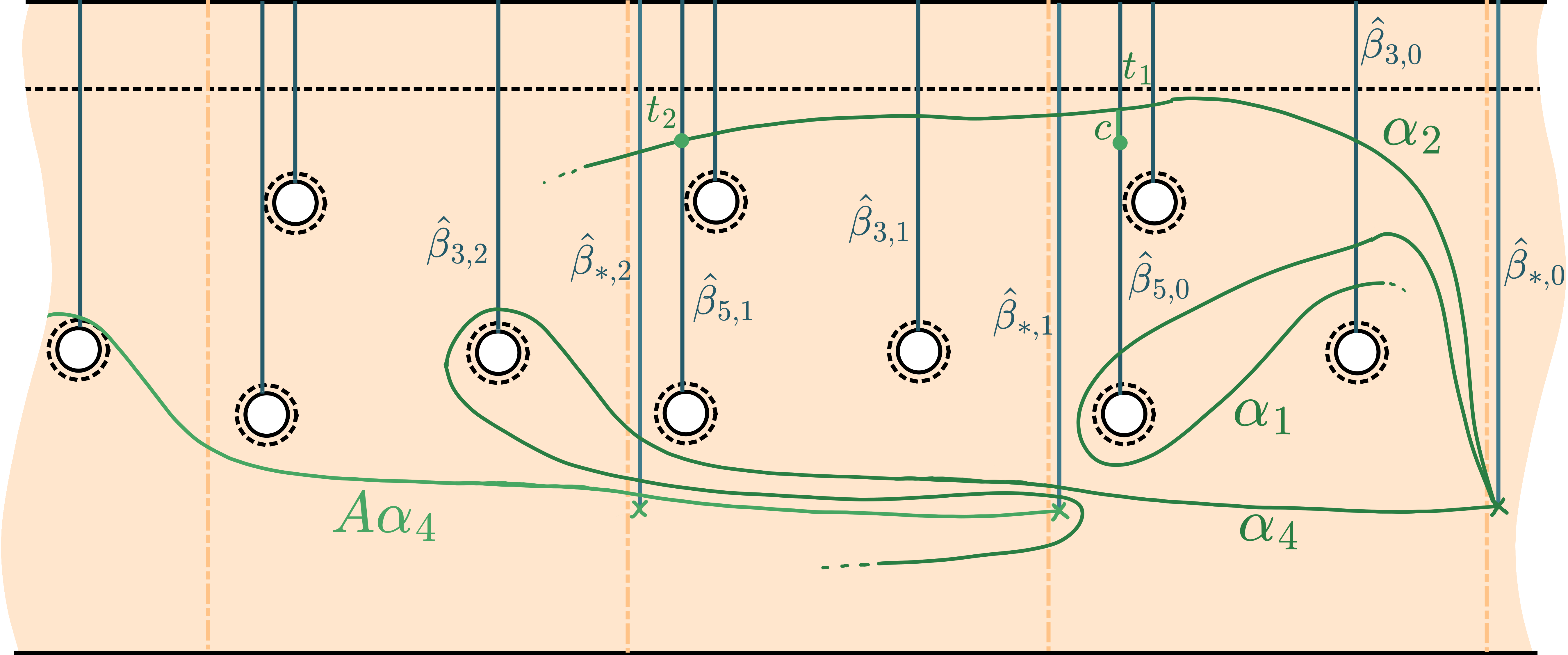}
            \caption{
                Some examples supporting Lemma~\ref{lem:eps_opt}.\\
                For statement \ref{enum:eps_opt_finite}, the curve $\alpha_1$ shows that there is some distance at least $\delta>0$ between the two crossings of $\hat{\beta}_{3,0}$.\\
                For statement \ref{enum:eps_opt_K}, the curve $\alpha_2$ shows the intersections with $\hat{\beta}_{5,0}$ at $t_1$ and $\hat{\beta}_{5,1}$ at $t_2$, with segment $c$ between $\alpha_2(t_1)$ and $A^{-1}\alpha_2(t_2)$.\\
                For statement \ref{enum:eps_opt_j_above}, the curve $\alpha_4$ and its other lift $A\alpha_4$ show that $\hat{\beta}_{*,1}$ is crossed at least twice, with the crossing of $\hat{\beta}_{*,2}$ in between, when $\hat{\beta}_{3,2}$ is crossed while $\hat{\beta}_{3,1}$ is not crossed.
            }
    \end{figure}

    We will prove the claims one-by-one:
    \begin{enumerate}
        \item \emph{$\alpha_{\hat{x}}$ intersects each $\hat{\beta}_{i,j}$ finitely many times.}\\
        Since $\alpha_{\hat{x}}$ is a geodesic, there is some minimal length $\delta>0$ between two intersections with $\hat{\beta}_{i,j}$. 
        Using equation \eqref{eq:eps_opt_Busemann}, we see that the Busemann function at the $(n+1)$-th intersection must be decreased by at least $n\delta-\epsilon$ with respect to the first intersection.
        Due to the bounded Busemann function on intersections, $n$ is finite.
        \item  \emph{If $\alpha_{\hat{x}}$ intersects $\hat{\beta}_{i,j}$ for some $i,j$, there exists a $K>0$ such that $\alpha_{\hat{x}}$ cannot intersect $\hat{\beta}_{i,j+k}$  for $k>K$.}\\
        Let $t_1,t_2$ be the intersection times with $\hat{\beta}_{i,j}$ resp. $\hat{\beta}_{i,j+k}$, and let $c$ be the segment of $\hat{\beta}_{i,j}$ between $\alpha_{\hat{x}}(t_1)$ and $A^{-k}\alpha_{\hat{x}}(t_2)$, and let $\ell(c)$ be its length.
        We claim that $\ell(c)$ is has an upper bound independent of $k$:

        We note that the Busemann function on $\hat{\beta}_{i,j}$ diverges toward both ideal endpoints, while 
        \begin{align}
            B_\gamma(\alpha_{\hat{x}}(t_1))<B_\gamma(\alpha_{\hat{x}}(0)),
        \end{align}
        which means that $\alpha_{\hat{x}}(t_1)$ is in an interval of $\hat{\beta}_{i,j}$ that is bounded independent of $k$.

        On the other hand $A^{-k}\alpha_{\hat{x}}(t_2)$ cannot be between $\alpha_{\hat{x}}(t_1)$ and (the lift of) the infinite boundary of $\partial_2$ to avoid self-intersections of the projections of $\alpha_{\hat{x}}$ to $\check{X}$, and it cannot be inside (the lift of) the funnel of $\partial_i$, since that would mean that $\alpha_{\hat{x}}(t_2)$ is in the funnel as well, which is impossible, since $\alpha_{\hat{x}}$ is a geodesic.

        We conclude that the length $\ell(c)$ is indeed bounded independent of $k$.
        \\\\
        Next, we consider the closed curve which we get from projecting $\alpha_{\hat{x}}|_{[t_1,t_2]}\cup c$ back to $\check{X}$. 
        This curve decomposes in $k$ separating curves, each not homotopic to $\partial_1$ and thus each of length at least $L_1+\delta$, for some $\delta>0$, independent of $k$. 
        This gives
        \begin{align}\label{eq:delta_t_wrt_lc}
            t_2-t_1-\epsilon>k(L_1+\delta)-\ell(c)-\epsilon.
        \end{align}  
        Using \eqref{eq:eps_opt_Busemann}, we get
        \begin{align}
            B_\gamma(\alpha_{\hat{x}}(t_1))-B_\gamma(A^{-k}\alpha_{\hat{x}}(t_2))+k L_1 >k(L_1+\delta)-\ell(c)-\epsilon,
        \end{align}
        which gives
        \begin{align}
            k<\frac{B_\gamma(\alpha_{\hat{x}}(t_1))-B_\gamma(A^{-k}\alpha_{\hat{x}}(t_2))+\ell(c)+\epsilon}{\delta}\leq \frac{2\ell(c)+\epsilon}{\delta},
        \end{align}
        where we used that the Busemann function is 1-Lipschitz.
        The right-hand side is now bounded independent of $k$ as claimed.
        
        \item \emph{For all $i$ there exists a $J_{\mathrm{min}}<0$ such that $\alpha_{\hat{x}}$ cannot intersect $\hat{\beta}_{i,j}$ for $j<J_{\mathrm{min}}$.}\\
        It follows immediately from the upper and lower bounds on Busemann function at intersections that there can be no intersections for $j<\frac{B_0-B_\gamma(\alpha_{\hat{x}}(0))}{L_1}$.

        \item \emph{For all $i$ there exists a $J_{\mathrm{max}}>0$ such that $\alpha_{\hat{x}}$ cannot intersect $\hat{\beta}_{i,j}$ for $j>J_{\mathrm{max}}$.}\\
        Here, the fact that $\alpha_{\hat{x}}$ projects to a non-self-intersecting curve becomes important.
        We start by assuming that $\alpha_{\hat{x}}$ intersects $\hat{\beta}_{i,j}$.
        For sufficiently large $j>0$, due to property \ref{enum:eps_opt_K}, there are some $\hat{\beta}_{i,j'}$ for $0\leq j'<j$ which it does not cross. 
        In particular, to avoid self-intersections of the projection, the curve $\alpha_{\hat{x}}$ has to start by `moving further away' from $\partial_1$ each turn around the cylinder, which means that we have to `trace back' crossing both $\hat{\beta}_{*,j}$ and $\hat{\beta}_{*,1}$ at least twice, with at least $2(j-1) L_1-\epsilon$ decrease in Busemann function between the two crossings of $\hat{\beta}_{*,1}$, due to \eqref{eq:eps_opt_Busemann}. 
        Again, since the Busemann function on intersections with $\hat{\beta}_{*,1}$ is bounded, $j$ has to be bounded from above. 
    \end{enumerate}

    Given these constraints, there are only finitely many $\hat{\beta}_{i,j}$ that $\alpha_{\hat{x}}$ can cross, and the number of crossings is also finite. 
    Since the complement is simply connected, there is thus only a finite number of homotopy classes that contain an $\epsilon$-optimal asymptotically parallel geodesic.
\end{proof}

\subsection{The spine}

We can define the spine analogous to \cite{Budd2025_Tree_bij}, where the base spirals take the role of the shortest paths to the root cusp, and prove the equivalents of \cite[Lemma~8]{Budd2025_Tree_bij} and \cite[Lemma~9]{Budd2025_Tree_bij}.

We define the spine $\Sigma$ of the extended half-tight cylinder $\check{X}$ as the set of all points in $\check{X}$ that have more than one base spiral,
\begin{align}
    \Sigma=\{x\in\check{X}|w(x)\geq2\}.
\end{align}
Inner vertices $V$ have at least $3$ base spirals:
\begin{align}
    V=\{x\in\check{X}|w(x)\geq3\}.
\end{align}

\begin{figure}
    \centering
    \begin{subfigure}{.3\textwidth}
        \centering
        \includegraphics[width=\textwidth]{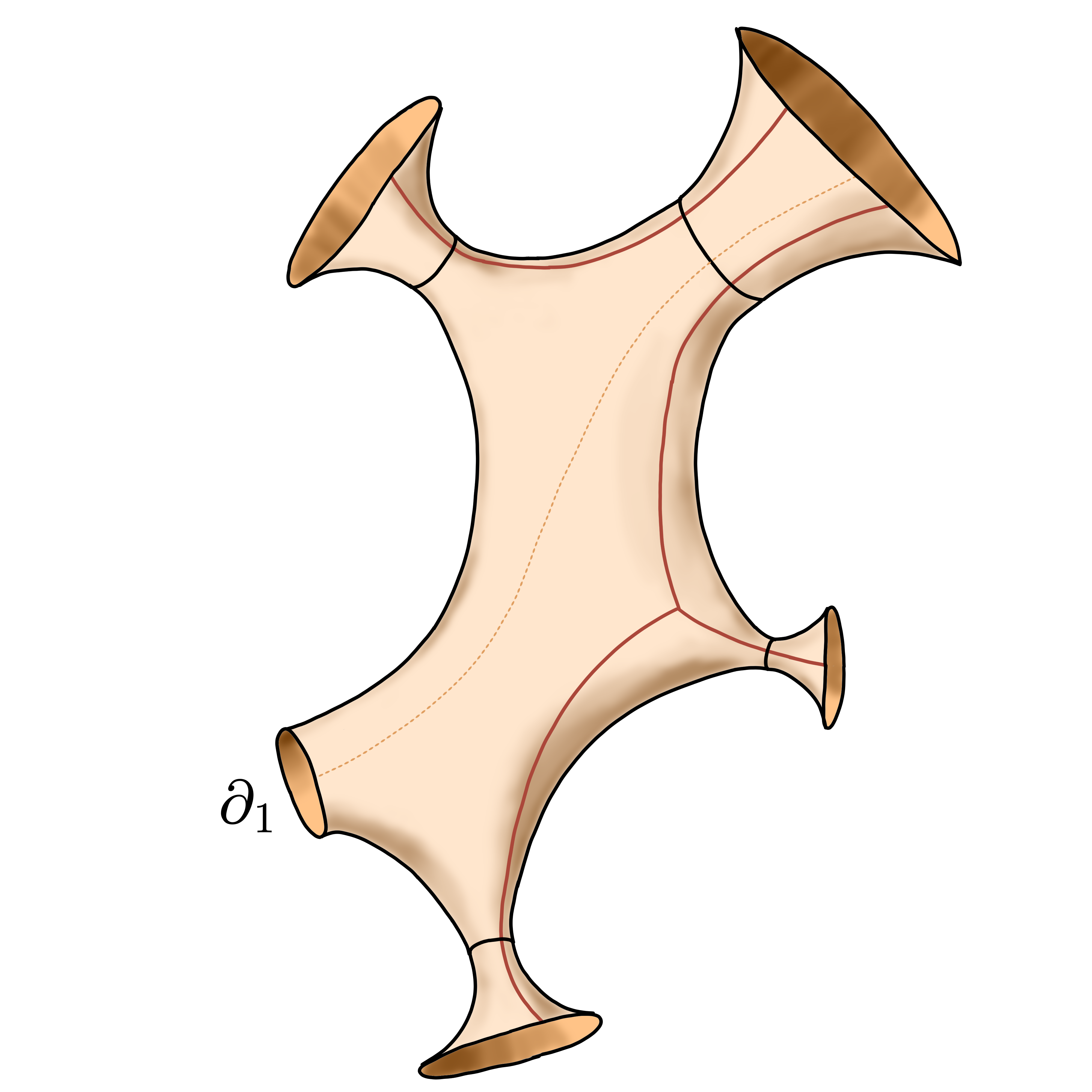}
    \end{subfigure}
    \begin{subfigure}{.65\textwidth}
        \centering
        \includegraphics[width=\textwidth]{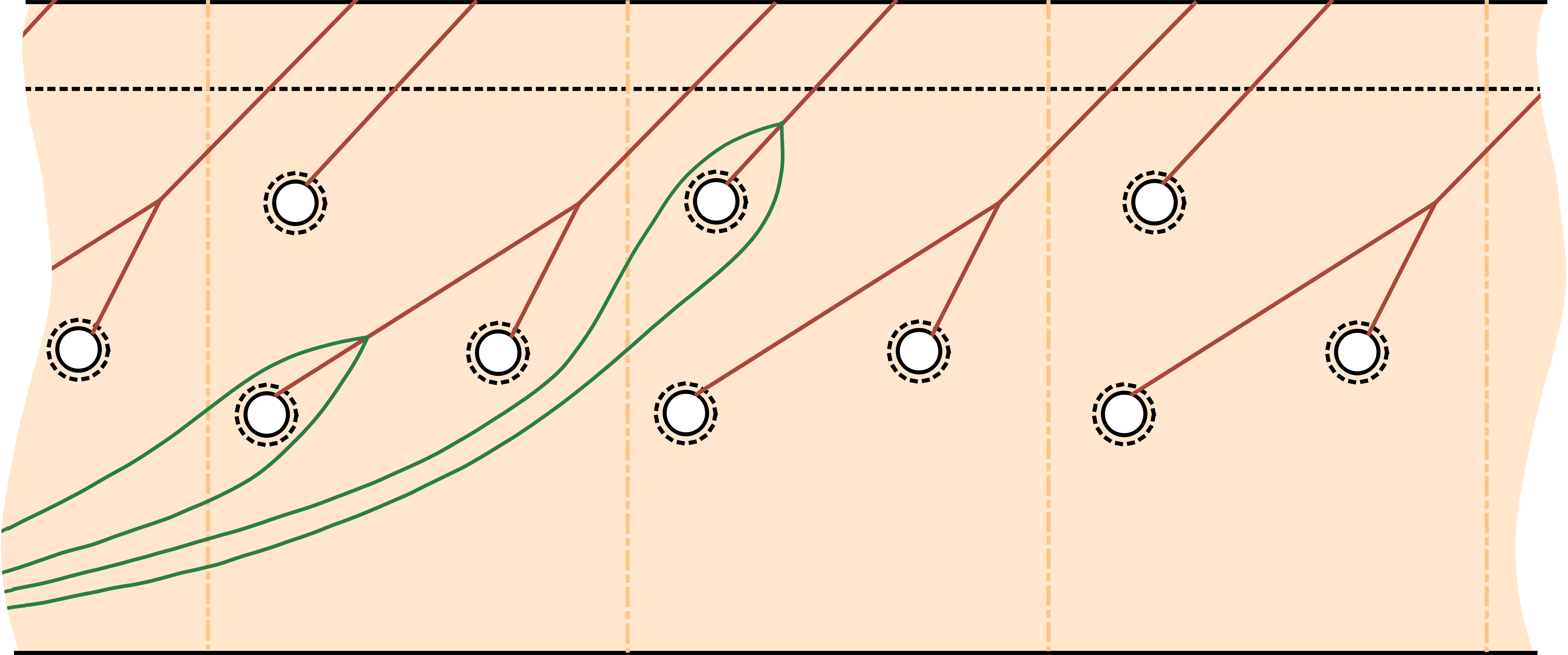}
    \end{subfigure}
    \caption{The spine, both in $\check{X}$ on the left and in $\hat{X}$ on the right. On the right also two examples of the multiple base spirals for points on the spine.}
\end{figure}

\pagebreak[2]

\begin{lemma}[{Extension of \cite[Lemma~8]{Budd2025_Tree_bij}}]\label{lem:spineproperties}
    The spine $\Sigma$ satisfies the following properties.
    \begin{enumerate}
        \item $V$ is a finite set.
        \item $\Sigma \setminus V = \{ x\in\check{X} : w(x) = 2\}$ consists of a finite union of open geodesic arcs. Tracking each arc in each direction it either ends at a point in $V$ or at an ideal point inside a funnel.
        \item Each point in $x\in V$ is the endpoint of exactly $w(x)\geq 3$ arcs.
        \item The boundary at infinity of each boundary or cusp (except the origin) contains the endpoint of at least one arc.
        \item The compactified spine $\overline{\Sigma}$, obtained by joining the endpoints of all arcs ending on the boundary at infinity of a single funnel or cusp, is a tree. 
    \end{enumerate}
\end{lemma}
\begin{proof}
    We adapt the arguments from \cite[Lemma~8]{Budd2025_Tree_bij} to our case.
    In that paper, the spine is defined as the locus of points with multiple shortest geodesics to a sufficiently small horocycle around a cusp. 
    The properties of the spine from this lemma are proven by considering the universal cover and the countable lifts of these horocycles.
    \\\\
    In our case, we have a similar setting:
    
    First we fix a parametrization for $\gamma$:
    For each point $x\in\check{X}$, we can pick a lift $\hat{x}\in\hat{X}$ and a base point $\gamma(0)$, such that $B_\gamma(x)>0$.
    
    Next, we consider the \emph{universal} cover $X^u$ of $\check{X}$, which is also the universal cover of $\hat{X}$, we can consider the countably many ideal points corresponding to the limits $\lim_{t\to\infty}\gamma(t)$ and take the horocycles around these ideal point that intersect the corresponding lifts of $\gamma(0)$.

    Note that all asymptotically parallel geodesics in $\hat{X}$ are lifted in $X^u$ to geodesics that end in one of these ideal points, and that the lengths of these spirals is given by the distance of the starting point to the corresponding horocycle.
    We thus have the same setting as the proof of \cite[Lemma~8]{Budd2025_Tree_bij}: we are interested in the number of shortest geodesics to a collection of horocycles in the universal cover.
    
    Unlike \cite{Budd2025_Tree_bij}, the horocycles in our case are not disjoint, but we can use Lemma~\ref{cor:finite_w} instead to still show that only finitely many horocycles are relevant.
\\\\
    The proof of the first four statements now follows from the same arguments as in \cite{Budd2025_Tree_bij}, which are based on the analysis at an $\epsilon$-neighborhood around the lift of $\hat{x}$ and the analysis at a single funnel.%
    \footnote{In \cite{Budd2025_Tree_bij} also horocyclic neighborhoods are considered, but they are not relevant here.}
    
    Finally, the final statement again follows from the fact that the complement of the spine can still be retracted to $\partial_1$.
\end{proof}
Note that this means that \cite[Lemma~9]{Budd2025_Tree_bij} also follows in our case, meaning that the spine can be seen as a tree $\tree\in \treeset_n^{\mathrm{all,HTC}}$, where corners are marked ideal if the adjacent arcs in $\Sigma$ are limiting parallel. 

\section{A canonical tiling}\label{sec:tiling}
Completely analogous to \cite{Budd2025_Tree_bij}, we can tile $\check{X}$ by triangles and wedges. 
We do this by noting that for any $x\in\check{X}\setminus\Sigma$, there is a unique base spiral starting at $x$, which we can uniquely extend away from the base, until it hits the spine at $\mathsf{end}(x)\in\Sigma$ or approaches an ideal point $\mathsf{end}(x)$. 
Together with the convention $\mathsf{end}(x)=x$ for any $x\in\Sigma$, this gives the partition
    \begin{align}
		\check{X}=\mathsf{Ribs}(\check{X})\cup\bigcup_{\edge\in \edgeset(\tree)} \hypdiamond_\edge \cup \bigcup_{\corner\in \cornerset(\tree)} \hypwedge_\corner\cup\, \partial_1,
	\end{align}
    where 
    \begin{align}
       \hypdiamond_\edge&=\{x\in\check{X}|\mathsf{end}(x)\in e\} \\
       \hypwedge_\corner&=\{x\in\check{X}|\mathsf{end}(x)\in c\} 
    \end{align}
    and where $\mathsf{Ribs}(\check{X})$ are the boundary geodesics of these quadrilaterals and wedges.
    See Figure~\ref{fig:tiling}.

\begin{figure}
    \centering
    \begin{subfigure}{.3\textwidth}
        \centering
        \includegraphics[width=\textwidth]{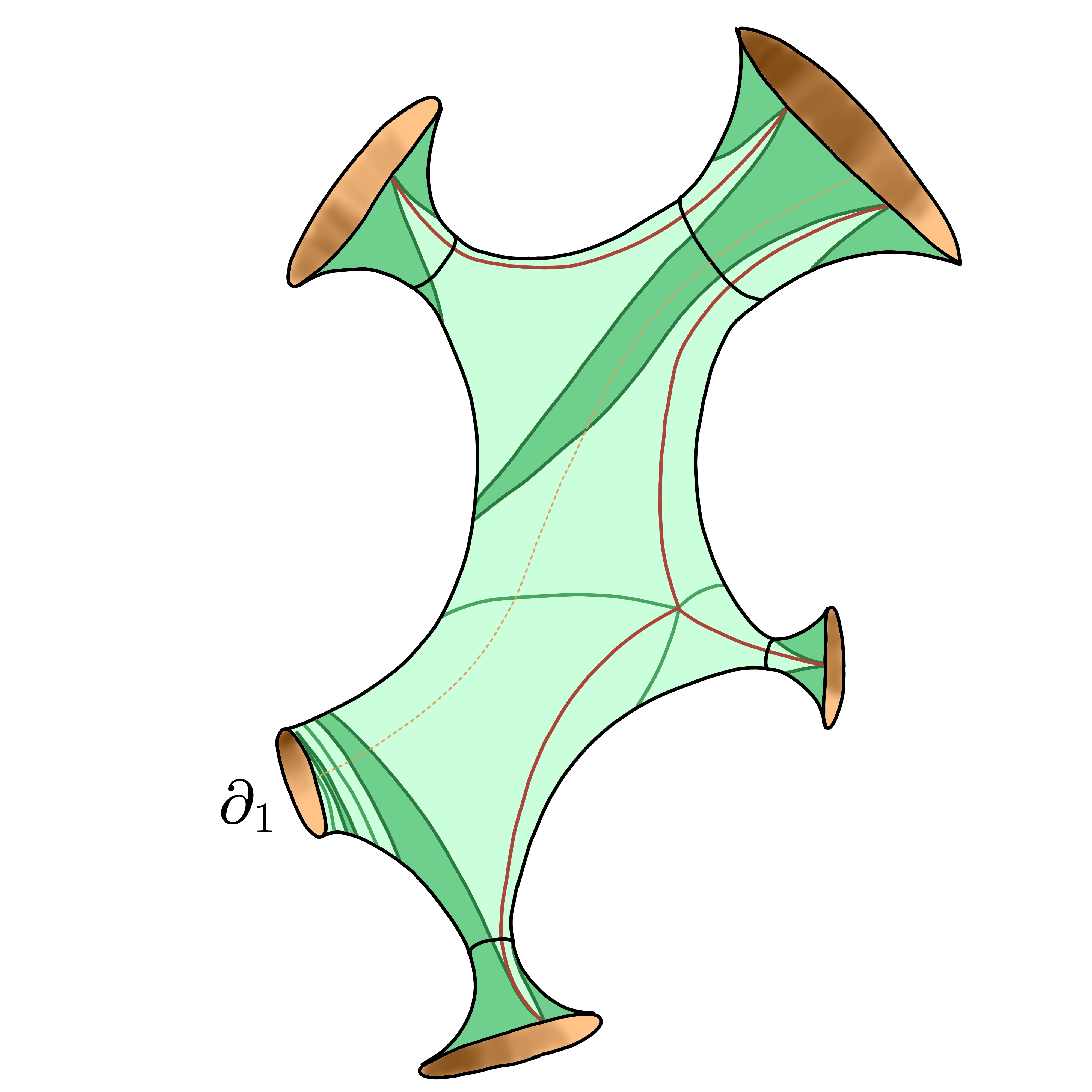}
    \end{subfigure}
    \begin{subfigure}{.65\textwidth}
        \centering
        \includegraphics[width=\textwidth]{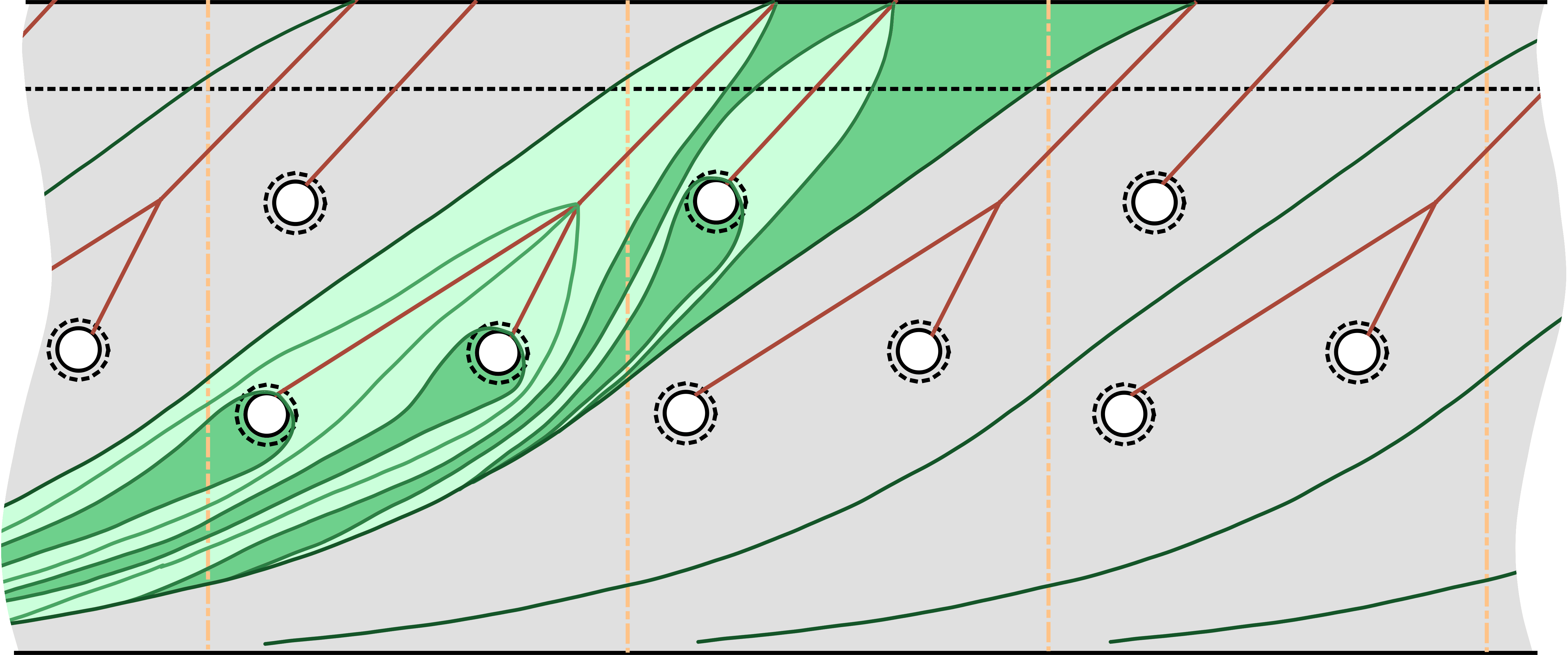}
    \end{subfigure}
    \caption{The tiling of $\check{X}$ on the left. We have wedges (darker) and quadrilaterals (lighter) in green. The boundaries of these tiles are the base spirals that end on an inner vertex or an ideal point of the spine. \\
    On the right we have chosen a slice in $\hat{X}$ by choosing one base spiral corresponding to an ideal endpoint of the spine on the infinite boundary of $\partial_2$. In this slice, we show the tiling.
    \label{fig:tiling}}
\end{figure}

Still analogous to \cite{Budd2025_Tree_bij}, we denote the angles in a quadrilateral $\hypdiamond_\edge$ by $\varphi:\vec{\edgeset}\to [0,\pi)$, for which the following properties hold:
\begin{enumerate}[label = (\roman*)]
    \item $\varphi(\vec{\edge}) = 0$ if and only if $\vec{\edge}$ starts at a boundary vertex;
    \item $\varphi(\vec{\edge}) + \varphi(\cev{\edge}) < \pi$ for $\edge\in \edgeset(\tree)$;
    \item $\sum_{j=1}^{\deg(\vertex)} \varphi(\vec{\edge}_{\vertex,j}) = \pi$ for $\vertex \in \vsetinner(\tree)$.
\end{enumerate} 

We also introduce around a boundary vertex $\bvertex\in\vsetboundary(\tree)$ the parameters $w_{\bvertex,i}$ for each adjacent edge and $v_{\bvertex,i}$ for each non-ideal corner. 
As in \cite{Budd2025_Tree_bij}, $e^{w_{\bvertex,i}}$ and $e^{v_{\bvertex,i}}$ are given by the quotient of the radii of the corresponding ribs, when we project the surface in the upper half plane, such that the corresponding original boundary in $X$ is vertical.

Note that we still have
\begin{align}
	\exp\left(\sum_{i=1}^{\deg(\bvertex)} w_{\bvertex,i} + \sum_{i=1}^{\operatorname{nonid}(\bvertex)} v_{\bvertex,i}\right) = e^{L_\bvertex},
\end{align}
but the other property
\begin{align}
	\exp\left(2\sum_{i=1}^{\deg(\bvertex)} w_{\bvertex,i}\right) = e^{L_\bvertex}
\end{align}
is based on the matching of horocycles around the origin, which is no longer valid in our setting. 
To find the equivalent relation here, we need to see how the Busemann function can provide an alternative matching.

\subsection{The Busemann jump} 
To glue quadrilaterals and wedges around boundary vertices, we want to identify points whose base spirals have the same length. 
We need to be a bit more careful here, since this length is based on the Busemann function, which is only well-defined in $\hat{X}$, not in $\check{X}$. 

To circumvent this problem, we choose one of the non-ideal corners of $\bvertex_2$ in an arbitrarily deterministic way and consider the corresponding rib $\beta$. 
In the cylinder cover $\hat{X}$ we pick a fundamental domain bounded by two lifts $\hat{\beta}_1,\hat{\beta}_2$ of this rib, where $A\hat{\beta}_1=\hat{\beta}_2$.
We will call this fundamental domain a \emph{slice}, using the terminology from \cite[Section 7.2]{Bouttier2014_OnIrreducible_BGM_A}.

The Busemann function on $\check{X}$ can now be defined as the corresponding value within this slice, for a fixed base point $\gamma(0)$, which is continuous in $\check{X}$, except on $\beta$, where it will jump by $L_1$. 
We can now pick a sufficiently small level set of the Busemann function, which will correspond to taking horocycles in the quadrilaterals and wedges.
For adjacent quadrilaterals and wedges, these horocycles match, except at $\beta$.

\begin{figure}
    \centering
        \centering
        \includegraphics[width=.7\textwidth]{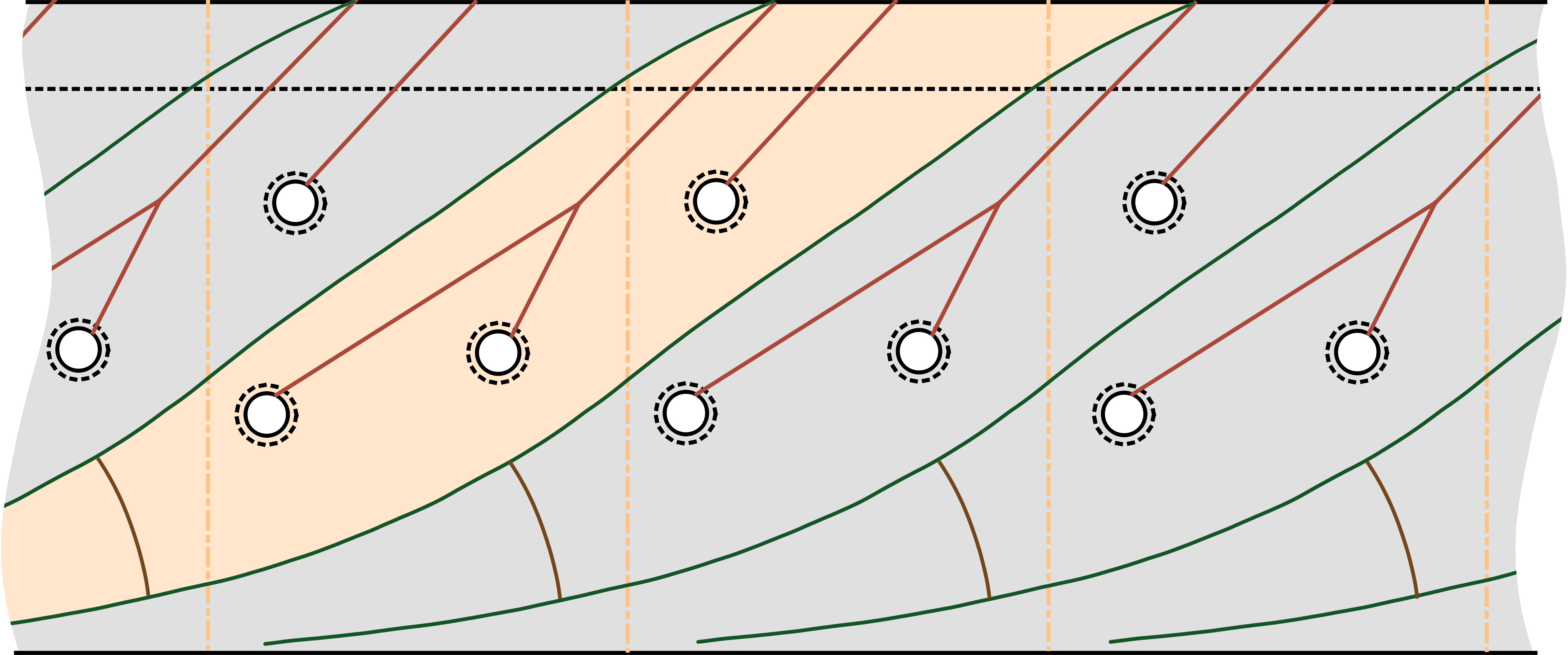}
        \caption{We have chosen a slice as before. The brown line is a line of constant Busemann function in this slice, which acts as a horocycle. Note that this line makes jumps of length $L_1$ when moving to an adjacent slice.}
\end{figure}

Following the same arguments of the proof of \cite[Lemma~12]{Budd2025_Tree_bij}, we get for $i>2$
\begin{align}
    \exp\left(\sum_{j=1}^{\deg(\bvertex)} w_{\bvertex,j}\right) = e^{L_i},
\end{align}
to match the horocycles, while for $i=2$, we need to incorporate the jump, giving
\begin{align}
    \exp\left(\sum_{j=1}^{\deg(\bvertex_2)} w_{\bvertex_2,j}\right) = e^{L_2-L_1}.
\end{align}

From this, the following result follows, which is out first major deviation from \cite{Budd2025_Tree_bij}
\begin{lemma}
    If $\bvertex \in \vsetboundary(\tree)$ is a boundary vertex with label $i>2$, then we have 
	\begin{align}
		\sum_{j=1}^{\deg(\bvertex)} w_{\bvertex,j} = \sum_{j=1}^{\operatorname{nonid}(\bvertex)} v_{\bvertex,j} = L_i/2.
	\end{align}

    Furthermore, for $i=2$ we have
    \begin{align}
		\sum_{j=1}^{\deg(\bvertex_2)} w_{\bvertex_2,j} &= (L_2-L_1)/2\\ 
        \sum_{j=1}^{\operatorname{nonid}(\bvertex_2)} v_{\bvertex_2,j} &= (L_2+L_1)/2.
	\end{align}
    
    These parameters uniquely determine the gluing of the quadrilaterals $\hypdiamond_{\vec{\edge}_{\bvertex,j}}$ and wedges $\hypwedge_{\corner_{\bvertex,j}}$ incident to $\bvertex$.
\end{lemma}

\subsection{Tree bijection}
We continue following the steps of \cite{Budd2025_Tree_bij}.

First we can define $\spinebijection(X)$ as taking the tree structure of the spine of $X$ and assigning the parameters
\begin{align}
    \left( (\varphi(\vec{\edge}))_{\vec{\edge}\in\vec{\edgeset}(\tree)}, \left((w_{\bvertex,j})_{j=1}^{\deg(\bvertex)},(v_{\bvertex,j})_{j=1}^{\operatorname{nonid}(\bvertex)}\right)_{\bvertex\in \vsetboundary(\tree)}\right) \in \polytope(\mathbf{L}).\label{eq:pointinpolytope}
\end{align}
As shown in \cite[Proposition~14]{Budd2025_Tree_bij} this makes $\spinebijection$ a well-defined injective mapping.

Similarly, we have the inverse mapping 
\begin{align}
    \mathsf{Glue} : \bigsqcup_{\tree \in \treeset^{\mathrm{all,HTC}}_n} \mathcal{A}_{\tree}(\mathbf{L}) \to \mathcal{H}_n(\mathbf{L}),
\end{align}
defined identically to \cite{Budd2025_Tree_bij}, where we use the tree structure and labels to construct the quadrilaterals and wedges and glue them accordingly.

It needs some more explanation to show that $\mathsf{Glue}$ is well-defined. 
Using arguments from \cite[Lemma~15]{Budd2025_Tree_bij}, it follows that the image is in $\mathcal{M}_{0,n}(\mathbf{L})$, but we need some work to show that it is a half-tight cylinder.

For this we need the equivalent of the `wrapping lemma' from \cite[Section 7.2]{Bouttier2014_OnIrreducible_BGM_A}, which essential to generalize the tree bijection to the half-tight cylinder.

Recall that we can pick a slice in $\hat{X}$ bounded by $\hat{\beta}_1$ and $\hat{\beta}_2=A\hat{\beta}_1$. 
In this slice, we can pick a line of constant (and sufficiently small) Busemann function, which acts as a horocycle $h$ around the ideal point at the bottom of this slice. 
We parametrize $\hat{\beta}_i$ by its signed distance to the horocycle.
\begin{lemma}\label{lem:min_dist}
    For any $t\in\R$, the distance between $\hat{\beta}_1(t)$ and $A\hat{\beta}_1(t)=\hat{\beta}_2(t-L_1)$ is bigger than $L_1$.
\end{lemma}
\begin{proof}
    \begin{figure}
        \centering
            \centering
            \includegraphics[width=.7\textwidth]{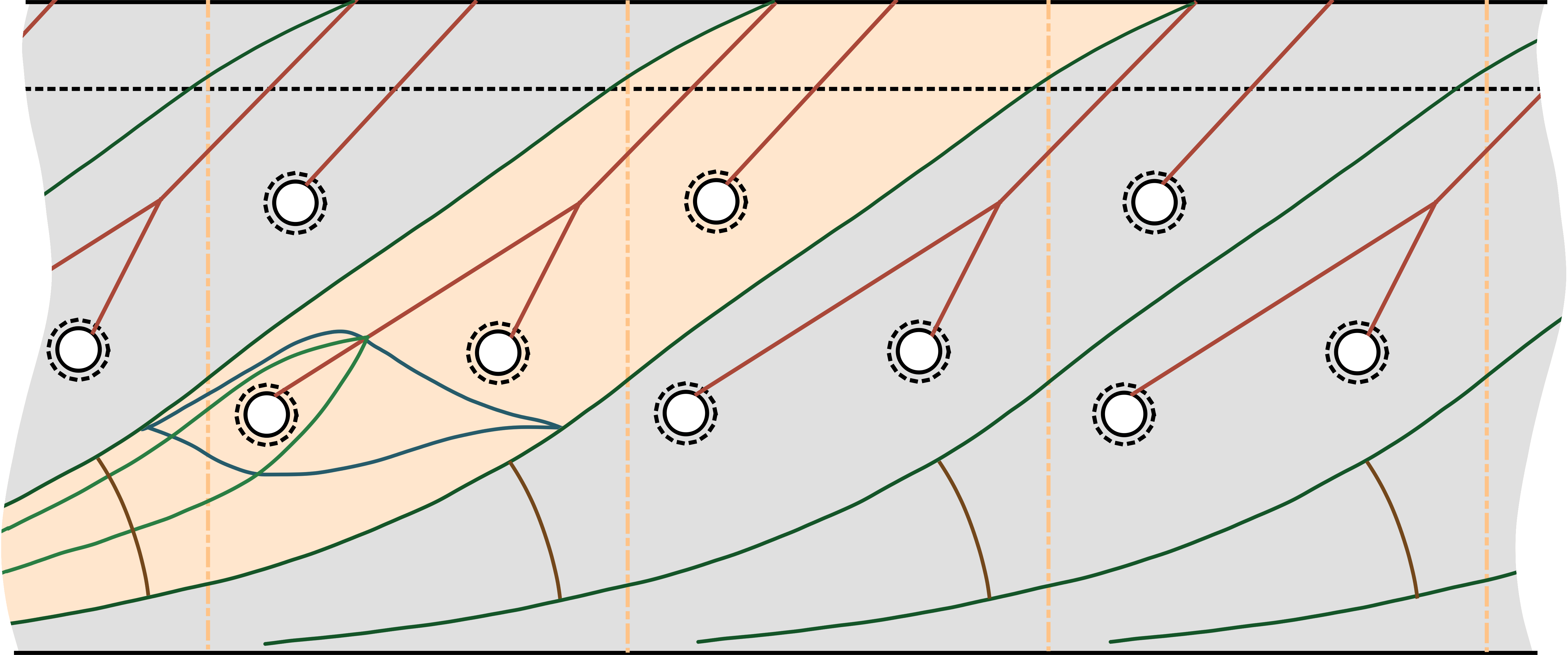}
            \caption{See Lemma~\ref{lem:min_dist}. We have in blue two examples of curves that, by this lemma, are longer than $L_1$. The lower curve is homotopic to the horocycle (with sliding endpoints on the base spirals), while the upper curve can be decomposed in two curves, which can be dealt with separately. }
    \end{figure}

    We will look at curves $\alpha$ that realize the distance between $\hat{\beta}_1(t)$ and $\hat{\beta}_2(t-L_1)$, and we identify two cases:
    \begin{itemize}
        \item Case 1: $\alpha$ does not cross any diagonal of the quadrilaterals or their endpoints.
        In this case $\alpha$ is homotopic to the horocycle $h$ and the result follows from the triangle inequality in the triangle with sided $\alpha$ and half-infinite segments of $\hat{\beta}_1$ and $\hat{\beta}_2$. 
        \item Case 2: $\alpha$ does cross $k$ times a diagonal of the quadrilaterals or their endpoints.
        In this case $\alpha$ is split in $k+1$ segments $\alpha_i$ at intersection points $x_i$ on diagonal $e_i$.
        We denote the distance within $\hypdiamond_{e_i}$ from $x_i$ to $h$ by $d_i$.  
        It follows from the gluing prescription that this distance is realized by geodesics in $\hypdiamond_{e_i}$ from $x_i$ to $h$ on both sides of the diagonal. 
        
        Again from the triangle inequality, it follows that the length of $\alpha_i$ is at least $|d_i-d_{i-1}|$, where $d_0=t-L-1$ and $d_{k+1}=t$, from which the claim follows. 
    \end{itemize}
\end{proof}
From this it is clear that the image of $\mathsf{Glue}$ is indeed a half-tight cylinder.

We can now prove the first claim of Theorem~\ref{thm:HTC_bij}
\begin{proposition}\label{prop:HTC_tree_bij}
    $\spinebijection$ is a bijection from $\mathcal{H}_n(\mathbf{L})$ to $\bigsqcup_{\tree \in \treeset^{\mathrm{all,HTC}}_n} \mathcal{A}_{\tree}(\mathbf{L})$, with $\mathsf{Glue}$ as inverse.
\end{proposition}
\begin{proof}
    By construction, it is clear that $\mathsf{Glue}\circ \spinebijection(X) = X$. 
    We thus only need to show that the reverse is also true. 
    In particular, we need to show that the union of diagonals of the quadrilaterals and inner vertices $\Sigma$ is indeed the spine of the image of $\mathsf{Glue}$.
    
    This again follows from considering a slice in the cylinder cover. 
    Within this slice, the arguments from \cite[{Proof of Theorem~7}]{Budd2025_Tree_bij} apply, meaning that the shortest geodesics from to $h$ are disjoint from (the lift of) $\Sigma$ and that there are only multiple such shortest geodesics starting at $\hat{x}$ if and only if $\hat{x}$ is on (the lift of) $\Sigma$.
    
    As discussed before, the distance to $h$ within a slice is equal to the Busemann function $B_\gamma$ for some choice of $\gamma(0)$. 
    This means that the shortest geodesics to $h$ in the slice are actually co-rays. 

    Finally, we need to check that there are no co-rays that cross multiple slices. 
    This is immediate from the fact that the slice is bounded by co-rays and co-rays cannot intersect.

    This means that $\Sigma$ is indeed the locus of points that have multiple co-rays, and thus the spine of the image of $\mathsf{Glue}$.
\end{proof}

\begin{proof}[Proof of Theorem~\ref{thm:HTC_bij}]
The first statement of Theorem~\ref{thm:HTC_bij} is exactly Proposition~\ref{prop:HTC_tree_bij}.

For the Weil--Petersson measure, we look at \cite[Section~3]{Budd2025_Tree_bij}. 
Remarkably, the arguments there only depend on the (ideal) triangulation and thus can be straightforwardly applied to our case: 
We have an ideal triangulation with the same relation between shear coordinates and tree labels $\lambda\in\polytope$, thus giving the same bivector on our polytopes.
The only difference is that our top-dimensional half-tight polytopes are different symplectic leaves, due to the different sum conditions for $(w_{\bvertex_2,i})$ and $(v_{\bvertex_2,i})$, but this does not influence the volume measure.

\end{proof}

\section{General surfaces}\label{sec:general}
We can now move on from half-tight cylinders to general surfaces with $L_1<L_2$.
We use a decomposition of the surface into half-tight cylinders, equivalent to \cite[Section~4.4]{Bouttier2024_OnQuasi_BGM_C} where this decomposition is done for maps.
On the level of Weil--Petersson volumes, this decomposition is also shown in \cite[Section~2.4]{Budd2024_Top_rec}.

\begin{figure}
    \centering
        \centering
        \includegraphics[scale=3]{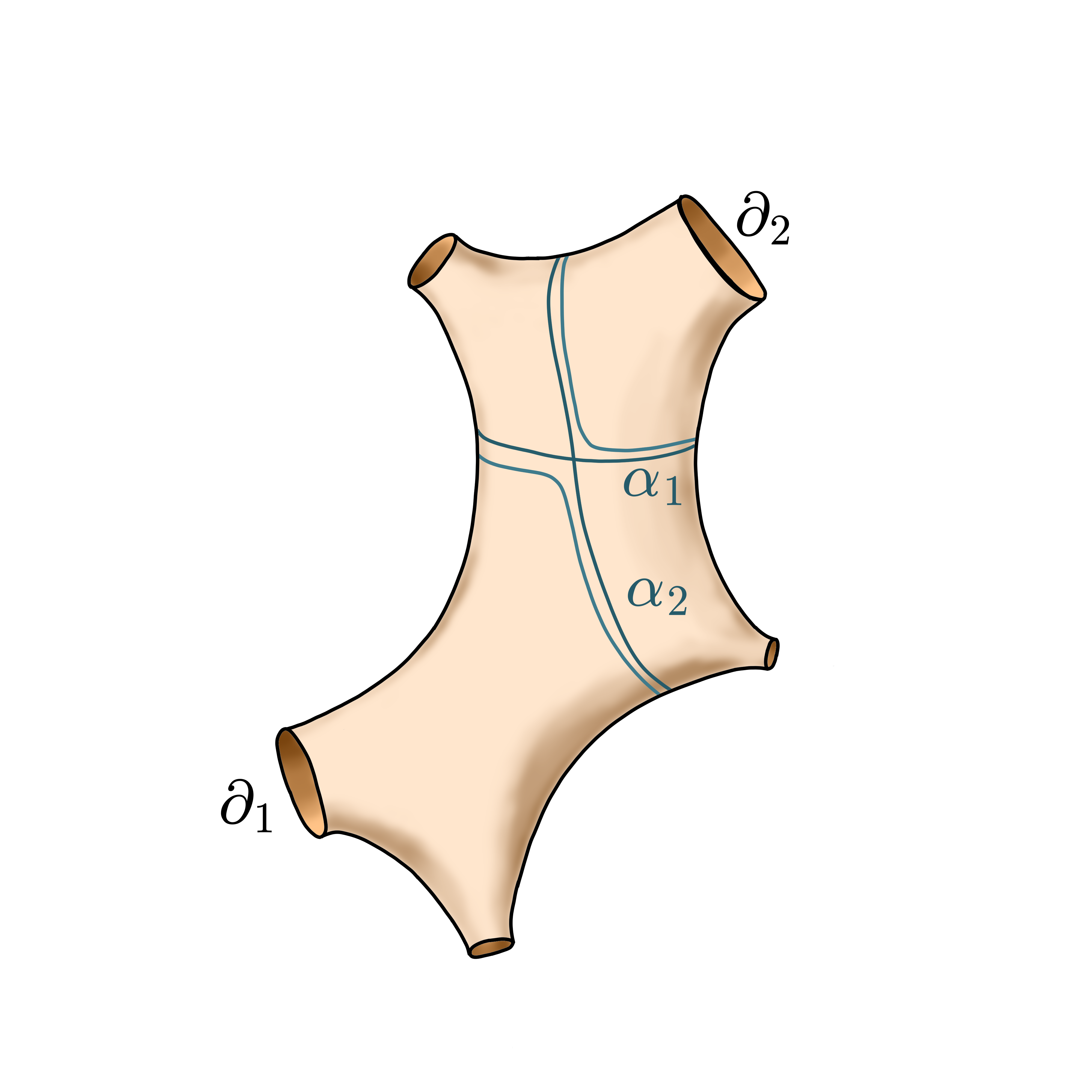}
        \caption{For Lemma~\ref{lem:unique_split}: The darker curves represent the case that $\alpha_1$ and $\alpha_2$, both of minimal length, intersect. The existence of the lighter curves, of which at least one will be shorter, gives a contradiction.\\
        For Lemma~\ref{lem:shortest_in_reglue}, we can use the same image, but one has to change the labels $\partial_1\mapsto\hat{\partial_2}$, $\alpha_1\mapsto\partial_1$ and $\alpha_2\mapsto\alpha$. 
        Again, this gives the two lighter curves, both contained in one half-tight cylinder, at least one of them shorter than $\partial_1$, which is again a contradiction.\label{fig:crossing_short_curves}}
\end{figure}

\begin{lemma}\label{lem:unique_split}
    For $L_1<L_2$, the subset $\mathcal{M}^\circ_{0,n}(\mathbf{L})\subset\mathcal{M}_{0,n}(\mathbf{L})$ that has a unique shortest closed geodesic separating $\partial_1$ and $\partial_2$, possibly being $\partial_1$ itself, is of full WP-measure.%
    \footnote{Be aware that $\mathcal{M}^\circ_{0,n}(\mathbf{L})$ is \emph{not} the equivalent of $\mathcal{M}^\circ_{0,1+n}(0,\mathbf{L})$ in \cite{Budd2025_Tree_bij}.}
\end{lemma}
\begin{proof}
    Due to the fact that there are only finite simple closed geodesics with lengths less or equal to $L_1$, there clearly exists a shortest closed geodesic separating $\partial_1$ and $\partial_2$.
    We are only left to prove that this curve is non-unique only in a zero WP-measure subset.

    Let $\alpha_1$ and $\alpha_2$ be two different shortest closed geodesic separating $\partial_1$ and $\partial_2$. 
    Note that these curves cannot intersect: if they intersect we can find strictly shorter separating curves, see Figure~\ref{fig:crossing_short_curves}.
    So $\alpha_1$ and $\alpha_2$ are disjoint curves, of which at least one of them is not $\partial_1$.
    The fact that the curves have the same lengths, makes the subset of $\mathcal{M}_{0,n}(\mathbf{L})$ for which this is the case of co-dimension $1$. 
    This means that the subset of $\mathcal{M}_{0,n}(\mathbf{L})$, where there is a unique shortest closed geodesic separating $\partial_1$ and $\partial_2$, is of full WP-measure.
\end{proof}

\begin{lemma}\label{lem:shortest_in_reglue}
    Let $X_1\in \mathcal{H}_{n_1}(\mathbf{L})$ and $X_2\in \mathcal{H}_{n_2}(\mathbf{\hat{L}})$ be two half-tight cylinders with boundaries $\partial_i$ and $\hat{\partial}_i$ respectively and $L_1=\hat{L}_1$. 
    Let $X\in\mathcal{M}_{0,n_1+n_2-2}(L_2,\dots,L_{n_1},\hat{L}_2,\dots,\hat{L}_{n_2})$ be a surface that one gets by gluing $\partial_1$ to $\hat{\partial}_1$.
    Then the glued curve $\partial_1$ is the unique shortest geodesic separating $\partial_2$ from $\hat{\partial}_2$.
\end{lemma}
\begin{proof}
    Let's assume that there is a curve $\alpha \neq \partial_1$ of length $\ell_\alpha \leq L_1$ that separates $\partial_2$ from $\hat{\partial}_2$. 
    Note that it cannot be fully in $X_1$ or in $X_2$, since that would violate tightness of $\partial_1$ or $\hat{\partial}_1$ respectively.
    This means that $\alpha$ intersects $\partial_1$. 
    We can now find two curves $\alpha_1$ and $\alpha_2$, which are in $X_1$ and $X_2$ respectively, separating $\partial_2$ from $\hat{\partial}_2$, where at least one has length $\ell_{\alpha_i} \leq L_1$, which we just stated to be impossible.
    See Figure~\ref{fig:crossing_short_curves}. 
    So we conclude that $\partial_1$ is the unique shortest geodesic separating $\partial_2$ from $\hat{\partial}_2$.
\end{proof}

\begin{proof}[Proof of Theorem~\ref{thm:Full_bij}]
    Note that due to Lemma~\ref{lem:unique_split}, for each surface $X\in\mathcal{M}^\circ_{0,n}(\mathbf{L})$, there is a unique shortest separating geodesic $\alpha$. 
    We consider two cases (similar to \cite[Section~2.4]{Budd2024_Top_rec}):
    
    Case 1: $\alpha=\partial_1$. 
    In this case we have a half-tight cylinder (possible after switching labels), which is according to Theorem~\ref{thm:HTC_bij} in bijection with $\bigsqcup_{\tree \in \treeset^{\mathrm{all,HTC}}_n} \mathcal{A}_{\tree}(\mathbf{L})$. 

    Case 2: $\alpha\neq\partial_1$.
    In this case we can cut the surface along $\alpha$, which gives us two half-tight cylinders. 
    These half-tight cylinders are each in bijection with $\bigsqcup_{\tree \in \treeset^{\mathrm{all,HTC}}_{|I_i|+1}} \mathcal{A}_{\tree}(\ell(\alpha),\mathbf{L}_{I_i})$, with $I_1\sqcup I_2=\{1,\dots,n\}$. 
    We can identify this with $\bigsqcup_{\dubtree \in \treeset^{\mathrm{all,full}}_n} \mathcal{A}_{\dubtree}(\mathbf{L})$ by adding $(\ell,\tau)$, where $\ell$ is the length of $\alpha$ and $\tau$ is the twist parameter that one needs to correctly glue back the half-tight cylinders.
\\\\
    The inverse mapping is straight-forward. 
    Here we also consider two cases:

    Case 1: We have $\tree \in \treeset^{\mathrm{all,HTC}}_n$.
    In this case we can use the bijection from Theorem~\ref{thm:HTC_bij} to get a half-tight cylinder with tight boundary $\partial_1$. 

    Case 2: We have $\dubtree \in \treeset^{\mathrm{all,full}}_n$.
    In this case we interpret the double trees as two trees in $\treeset^{\mathrm{all,HTC}}_{|I_i|+1}$ with corresponding labels. 
    The bijection from Theorem~\ref{thm:HTC_bij} gives us two half-tight cylinders, with tight boundaries of length $\ell$.
    We can glue the tight boundaries, using the twist $\tau$. 
    Lemma~\ref{lem:shortest_in_reglue} ensures that the glued geodesic is the unique shortest separating curve, which shows that this is indeed an inverse of the mapping above.
\\\\
    Finally, we consider the push-forward of the measure.
    In case 1, this just follows from Theorem~\ref{thm:HTC_bij}.
    In case 2, we use the fact that the WP-measure of the surface factorizes into $\dd{\ell}\dd{\tau}$ times the WP-measure of both half-tight cylinders \cite[Section~8]{Mirzakhani2006_SimpleGeodesics_TopRec}\cite[Section~2.3]{Budd2024_Top_rec}.  
\end{proof}

\section{Computing volumes}\label{sec:compute}
A clear consequence of Theorems~\ref{thm:HTC_bij} and \ref{thm:Full_bij} is that the Weil--Petersson volumes of genus-$0$ hyperbolic surfaces can be computed by summing over all (double) tree structures and integrating over all possible labels in the corresponding top-dimensional polytopes. 

There are some tricks we can use to simplify this integration, at the cost of losing some bijective interpretation.
The first tricks are straight-forward adaptations from the tricks in \cite[Section~4]{Budd2025_Tree_bij}.

\subsection{Half tight cylinder}
The important thing to note is that many labels are independent.
Especially the labels incident to a boundary vertex can be straight-forwardly integrated.
For a boundary vertex $\bvertex\neq\bvertex_2$, we get a volume of the polytope $\mathcal{A}_{\bvertex}$
\begin{align}
    \int_{0}^{L_\bvertex/2}\prod_{j=1}^{\deg{\bvertex}}\dd{w_j}\dd{v_j}\, \delta\left(L_\bvertex/2-\sum_{j=1}^{\deg{\bvertex}}w_j\right)\,\delta\left(L_\bvertex/2-\sum_{j=1}^{\deg{\bvertex}}v_j\right)=\left(\frac{(L_\bvertex/2)^{\deg{\bvertex}-1}}{(\deg{\bvertex}-1)!}\right)^2.
\end{align} 
This corresponds to 
\begin{align}
    \frac{t_{\deg{\bvertex}-1}(L_\bvertex)}{2(\deg{\bvertex}-1)!},
\end{align}
where
\begin{align}
    t_k(L)=\frac{2}{k!}\left(\frac{L}{2}\right)^{2k}.
\end{align}
The factor of two here is to match conventions, but can also be seen as (partly) incorporating the total factor $2^{n-3}$.
Similarly, the factor $(\deg{\bvertex}-1)!$ can be seen as counting the ordering of the edges around the vertex.

Similarly, we get for $\bvertex_2$
\begin{align}
    &\int_{0}^{\infty}\prod_{j=1}^{\deg{\bvertex_2}}\dd{w_j}\dd{v_j}\, \delta\left((L_2-L_1)/2-\sum_{j=1}^{\deg{\bvertex_2}}w_j\right)\,\delta\left((L_2+L_1)/2-\sum_{j=1}^{\deg{\bvertex_2}}v_j\right) \nonumber\\
&\mkern300mu =\frac{((L_2-L_1)/2)^{\deg{\bvertex_2}-1}}{(\deg{\bvertex_2}-1)!}\frac{((L_2+L_1)/2)^{\deg{\bvertex_2}-1}}{(\deg{\bvertex_2}-1)!},
\end{align}
such that the volume of $\mathcal{A}_{\bvertex_2}$ is given by
\begin{align}
    \frac{\tilde{t}_{\deg{\bvertex_2}-1}(L_2,L_1)}{2(\deg{\bvertex_2}-1)!},
\end{align}
with
\begin{align}
    \tilde{t}_k(L,\ell)=2\frac{(L^2-\ell^2)^k}{4^k k!}\ind_{\ell<L},\quad\quad k\geq0.
\end{align}

For inner vertices, we have to consider the Delaunay condition. 
As explained in \cite[Section~4]{Budd2025_Tree_bij}, we can use the principle of inclusion/exclusion to change this to anti-Delaunay conditions and contract anti-Delaunay subtrees to single vertices.

We can integrate the labels of an anti-Delaunay subtree $\tree_\vertex$, giving a volume for the polytope $\mathcal{A}_{\vertex}$ of $\gamma_{\deg{\vertex}-1}/(\deg{\vertex}-1)!$, where $\deg{\vertex}$ is the number of leaves of the anti-Delaunay subtree $\tree_\vertex$ and $\gamma_k$ is given by
\begin{align}
    \gamma_k=(-1)^k \frac{\pi^{2k-2}}{(k-1)!}.
\end{align}
See \cite{Budd2025_Tree_bij} for the derivation.

We get 
\begin{align}\label{eq:HTC_vol_from_tree}
    H_n(\mathbf{L})=\frac{1}{4}\sum_{\tree \in \widetilde{\treeset}^{\mathrm{HTC}}_n} \frac{\tilde{t}_{\deg{\bvertex_2}-1}(L_2,L_1)}{(\deg{\bvertex_2}-1)!}
    \prod_{\substack{\bvertex\in\vsetboundary(\tree)\\\bvertex\neq\bvertex_2}} \frac{t_{\deg{\bvertex}-1}(L_i)}{(\deg{\bvertex}-1)!} 
    \prod_{\vertex\in \vsetinner(\tree)}\frac{\gamma_{\deg{\vertex}-1}}{(\deg{\vertex}-1)!},
\end{align}
where we note that $\widetilde{\treeset}^{\mathrm{HTC}}_n$ is the same set as $\treeset^{\mathrm{all,HTC}}_n$.
Note that is expression is very similar to the expression in \cite[Theorem~3]{Budd2025_Tree_bij}, except for $b_2$.

\begin{proof}[Proof of Theorem~\ref{thm:HTC_gf}] 
    Note that the same $R[\mu]$ has been computed using the tree bijection in \cite{Budd2025_Tree_bij}, where $R[\mu]$ is the generating function of WP-volumes of hyperbolic surfaces with two cusps. 
    Comparing the tree bijections, in particular \cite[Equation~(66)]{Budd2025_Tree_bij} with \eqref{eq:HTC_vol_from_tree}, we can see that in our case $2R[\mu]$ is the generating function of volume contributions in \eqref{eq:HTC_vol_from_tree} corresponding to rooted subtrees of $\tree\in\widetilde{\treeset}^{\mathrm{HTC}}_n$ that do not include the special boundary vertex $b_2$.
    
    We can now easily derive the volume corresponding to any tree $\tree\in\widetilde{\treeset}^{\mathrm{HTC}}_n$, where the boundary vertex $b_2$ is incident to $k+1 \geq 1$ edges.
    We get
    \begin{align}
        H(L_1,L_2;\mu]&=\frac{1}{4}\sum_{k=0}^\infty \frac{\tilde{t}_{k}(L_2,L_1)}{k!} \frac{(2R[\mu])^k}{k+1}\\
        &=\sum_{k=0}^\infty \frac{2^{-k} R[\mu]^{k+1}}{k!(k+1)!}(L_2^2-L_1^2)^{k},   
    \end{align}
    where the factor $1/(k+1)$ accounts for the cyclic permutations around $b_2$.
\end{proof}

\subsection{General surfaces}\label{ssec:volumes_general}
We can use the same tricks for the double trees.
Still assuming $L_1<L_2$, we get
\begin{align}
    V_{0,n}(L_1,\cdots,L_n)&=\frac{1}{4}\sum_{\tree \in \widetilde{\treeset}^{\mathrm{HTC}}_n} \frac{\tilde{t}_{\deg{\bvertex_2}-1}(L_2,L_1)}{(\deg{\bvertex_2}-1)!}
    \prod_{\substack{\bvertex\in\vsetboundary(\tree)\\\bvertex\neq\bvertex_2}} \frac{t_{\deg{\bvertex}-1}(L_i)}{(\deg{\bvertex}-1)!} 
    \prod_{\vertex\in \vsetinner(\tree)}\frac{\gamma_{\deg{\vertex}-1}}{(\deg{\vertex}-1)!}\nonumber\\&\quad
     + \frac{1}{16}\sum_{\dubtree\in\widetilde{\treeset}^{\mathrm{full}}_n} \int_{0}^{\infty} \ell \dd{\ell}\frac{\tilde{t}_{\deg{\bvertex_1}-1}(L_1,\ell)}{(\deg{\bvertex_1}-1)!}\frac{\tilde{t}_{\deg{\bvertex_2}-1}(L_2,\ell)}{(\deg{\bvertex_2}-1)!}
     \prod_{\substack{\bvertex\in\vsetboundary(\dubtree)\\\bvertex\neq\bvertex_1,\bvertex_2}} \frac{t_{\deg{\bvertex}-1}(L_\bvertex)}{(\deg{\bvertex}-1)!} 
    \prod_{\vertex\in \vsetinner(\dubtree)}\frac{\gamma_{\deg{\vertex}-1}}{(\deg{\vertex}-1)!}.
\end{align}
Note that we already integrated over $\tau$.
To combine the two terms, we define
\begin{align}
    \tilde{t}_{-1}(L,\ell)=4\frac{\delta(L-\ell)}{\ell}.
\end{align}
Furthermore, we can turn the trees in $\widetilde{\treeset}^{\mathrm{HTC}}_n$ into double trees $\widetilde{\treeset}^{\mathrm{HTC}*}_n$, by adding a disconnected boundary vertex $\bvertex_1$ and denote
\begin{align}
    \treeset^{\mathrm{graph}}_n=\left(\widetilde{\treeset}^{\mathrm{HTC}*}_n\sqcup \widetilde{\treeset}^{\mathrm{full}}_n\right) / \sim,
\end{align}
where two double trees are equivalent when they are equivalent as graphs, so forgetting the planar ordering at the vertices.

We get
\begin{align}
    V_{0,n}(L_1,\cdots,L_n)&=
     \frac{1}{16}\sum_{\dubtree\in\treeset^{\mathrm{graph}}_n} \int_{0}^{\infty} \ell \dd{\ell}\,\tilde{t}_{\deg{\bvertex_1}-1}(L_1,\ell)\tilde{t}_{\deg{\bvertex_2}-1}(L_2,\ell)\prod_{\substack{\bvertex\in\vsetboundary(\dubtree)\\\bvertex\neq\bvertex_1,\bvertex_2}} t_{\deg{\bvertex}-1}(L_\bvertex) 
    \prod_{\vertex\in \vsetinner(\dubtree)}\gamma_{\deg{\vertex}-1}.
\end{align}

We can preform the integration over $\ell$.
Assuming $L_1<L_2$, a straightforward calculation shows that for $a\geq -1 $ and $b\geq0$
\begin{align}
    \int_{0}^{\infty} \ell \dd{\ell}\,\tilde{t}_{a}(L_1,\ell)\tilde{t}_{b}(L_2,\ell)=2\sum_{m=0}^b (-1)^m t_{a+1+m}(L_1)t_{b-m}(L_2),
\end{align}
which yields
\begin{align}\label{eq:graph_sum}
    V_{0,n}(L_1,\cdots,L_n)=\frac{1}{8} \sum_{\dubtree\in\treeset^{\mathrm{graph}}_n}\sum_{m=0}^{\deg{\bvertex_2}-1} (-1)^m t_{\deg{\bvertex_1}+m}(L_1)t_{\deg{\bvertex_2}-1-m}(L_2)\prod_{\substack{\bvertex\in\vsetboundary(\dubtree)\\\bvertex\neq\bvertex_1,\bvertex_2}} t_{\deg{\bvertex}-1}(L_\bvertex) \prod_{\vertex\in \vsetinner(\dubtree)}\gamma_{\deg{\vertex}-1}.
\end{align}

\begin{figure}
    \centering
    \begin{subfigure}[t]{.45\textwidth}
        \centering
        \includegraphics[scale=1]{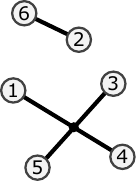}
        \caption{}
    \end{subfigure}
    \hspace{.05\textwidth}
    \begin{subfigure}[t]{.45\textwidth}
        \centering
        \includegraphics[scale=1]{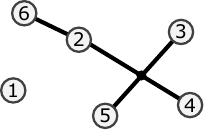}
        \caption{}
    \end{subfigure}
    \caption{Two examples of double trees in $\treeset^{\mathrm{graph}}_6$.
    The tree is colored black to stress that we neglect planar ordering at the vertices, in contrast to the red trees before.
    \\
    The left tree is not in $\treeset^{2\leftrightarrow3}_6$, since $\bvertex_2$ is not connected to $\bvertex_3$, while the right tree in $\treeset^{2\leftrightarrow3}_6$.
    \\
    Note that these double trees are related: we can move the subtree that contains $\bvertex_3$ between $\bvertex_1$ and $\bvertex_2$. 
    As a consequence, all but one term in \eqref{eq:graph_sum} regarding these double tree cancel.
    Only the $m=0$ term of the right tree remains. }
\end{figure}

We can even simplify further, by noting that we can cancel the $m=i$ term from double trees where $\bvertex_1$ and $\bvertex_3$ are in the same tree, against the $m=i+1$ term from trees where $\bvertex_1$ and $\bvertex_3$ are in different trees, leaving only the $m=0$ term in the latter trees.
This then gives
\begin{align}\label{eq:vol_weights_final}
    V_{0,n}(L_1,\cdots,L_n)=\frac{1}{8} \sum_{\dubtree\in\treeset^{2\leftrightarrow3}_n} t_{\deg{\bvertex_1}}(L_1)\prod_{\substack{\bvertex\in\vsetboundary(\dubtree)\\\bvertex\neq\bvertex_1}} t_{\deg{\bvertex}-1}(L_\bvertex) \prod_{\vertex\in \vsetinner(\dubtree)}\gamma_{\deg{\vertex}-1},
\end{align}
where the sum is now over all pairs of (combinatorial) trees $(\tree_1,\tree_2)\in\treeset^{2\leftrightarrow3}_n\subset\treeset^{\mathrm{graph}}_n$, consisting of $n$ boundary vertices $\bvertex\in\vsetboundary(\dubtree)$ with label $i=1,\ldots,n$ of arbitrary degree $\deg{\bvertex}$ and an arbitrary number of unlabeled inner vertices $\vertex\in \vsetinner(\dubtree)$ of arbitrary degree $\deg{\vertex}\geq3$, such that $\bvertex_1\in\tree_1$ and $\bvertex_2,\bvertex_3\in\tree_2$.

\begin{proposition}\label{prop:tree_gf}
    The polynomials $f_n(\hat{t}_0,\cdots,\hat{t}_{n-3};1/\hat{\gamma}_1,\hat{\gamma}_2,\cdots,\hat{\gamma}_{n-2})$ from Theorem~\ref{thm:recursion} are the generating functions of pairs of trees $\dubtree\in\treeset^{2\leftrightarrow3}_n$, where we add an extra half-edge to $\bvertex_1$, such that its degree increases by $1$. 
    In this generating function $\hat{t}_k$ counts the boundary vertices of degree $k+1$, $\hat{\gamma}_k$ counts the inner vertices of degree $k+1$ and the edges are counted by $-1/\hat{\gamma}_1$.   
\end{proposition}
Note that Theorem~\ref{thm:recursion} follows directly from this proposition, together with \eqref{eq:vol_weights_final}.
The assumption $L_1<L_2$ can be released here due to symmetry and continuity of the Weil--Petersson volumes.
\begin{proof}[Proof of Proposition~\ref{prop:tree_gf}]
    We give a proof by induction.
    Note that the base case is straight forward, given that there is only a single element in $\treeset^{2\leftrightarrow3}_3$, which has one edge connecting boundary vertices $\bvertex_2$ and $\bvertex_3$.
    
\begin{figure}
    \centering
        \includegraphics[width=\textwidth]{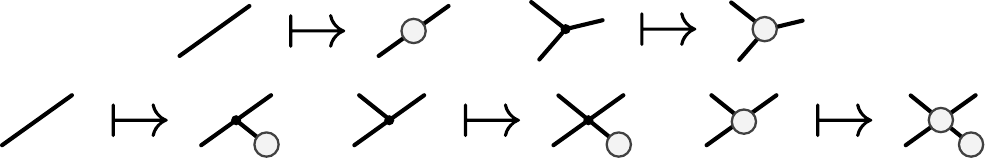}
        \caption{The five options to add a $(n+1)$-th boundary to an existing double tree in $\treeset^{2\leftrightarrow3}_n$. \label{fig:5options}}
\end{figure}

    For the induction step, we note that we can build the trees in $\treeset^{2\leftrightarrow3}_{n+1}$ by adding a $(n+1)$-th boundary vertex to the trees in $\treeset^{2\leftrightarrow3}_n$ using one of the following five distinct options, see Figure~\ref{fig:5options}:
    \begin{itemize}
        \item We can add the new boundary vertex in the middle of an edge. 
                This gives a contribution to $f_{n+1}$ of $\hat{t}_1\pdv{\hat{\gamma}_1}f_n$;
        \item We can replace an inner vertex by the new boundary vertex.
                This gives a contribution to $f_{n+1}$ of $\sum_{k=1}^{n-3}\hat{t}_{k+1}\pdv{\hat{\gamma}_{k+1}}f_n$;
        \item We can attach the new boundary vertex to another boundary vertex, using an edge.
                This gives a contribution to $f_{n+1}$ of $-\sum_{k=0}^{n-3}\hat{t}_{0}\hat{t}_{k+1}/\hat{\gamma}_1 \pdv{\hat{t}_{k}}f_n$;
        \item We can attach the new boundary vertex to an inner vertex, using an edge.
                This gives a contribution to $f_{n+1}$ of $-\sum_{k=1}^{n-3}\hat{t}_{0}\hat{\gamma}_{k+2}/\hat{\gamma}_1 \pdv{\hat{\gamma}_{k+1}}f_n$;
        \item We can attach the new boundary vertex to an edge, creating a new inner vertex.
                This gives a contribution to $f_{n+1}$ of $-\hat{t}_{0}\hat{\gamma}_{2}/\hat{\gamma}_1 \pdv{\hat{\gamma}_{1}}f_n$.
    \end{itemize} 
    Note that these contributions exactly add to \eqref{eq:f_rec}, which completes the proof.
\end{proof}

\section{Outlook}\label{sec:outlook}
The tree bijection in this paper extends the tree bijection from \cite{Budd2025_Tree_bij}, but doesn't yet cover all hyperbolic surfaces. 
In particular, the tree bijection only works for $g=0$. 
Perhaps surprisingly, in \cite{Budd2024_Top_rec} it is shown that the recursion from Theorem~\ref{thm:recursion} does also hold for higher genera, only with a different base case. 
This suggests that this tree bijection might be adapted to this case, where the image would not be a tree anymore, but some other map.

Another direction would be to include cone points, in combination with or instead of geodesic boundaries. 
In an upcoming paper \cite{Budd2026_Cones} together with Budd, we will in fact extend the tree bijection to planar hyperbolic surfaces with only cones and at least one cusp, similar to the tree bijection in \cite{Budd2025_Tree_bij}. 
It is an open question if also in this case the requirement of a single cusp can be relaxed, potentially with similar methods as in this paper.

Furthermore, in \cite{Budd2024_Top_rec} hyperbolic surfaces where a set of boundaries are tight are analyzed. 
Is this tightness property something that we can easily deduce from the tree labels?
\\\\
Also, in section~\ref{ssec:volumes_general} we used some cancellations between double trees, which resulted in keeping only the trees where boundaries 2 and 3 were connected. 
This nice cancellation suggests that there might be a more direct way to get to this result. 
Can we adapt the splitting of the surface, such that we immediately get that boundaries 2 and 3 are on the same tree? 
An answer to this question might even help to adapt this bijection to the more general cases above.

Due to fact that boundary 3 gets a more important role, we might be searching for an equivalent to \cite{Bouttier2022_Bijective_BGM_B}, where the authors use slices and Busemann functions on pairs of pants. 
\\\\
Finally, in \cite{Budd2025_Tree_bij} the tree bijection is used to get some global distance statistics. 
Can we find similar statistics in our setting, even though the spine is defined with asymptotically parallel geodesics?
In \cite{budd2025_randompuncturedhyperbolicsurfaces} similar statistics are used to find a scaling limit of uniform random hyperbolic surfaces with increasing number of cusps.
Can we also get this (or other) scaling limit for geodesic boundaries?

\clearpage
\bibliographystyle{siam}
\bibliography{busemann}

\end{document}